\documentclass[a4paper]{amsart}
\usepackage[utf8]{inputenc}
\usepackage{amsmath,amsthm,graphicx,tikz}
\usepackage[hyperfootnotes=false]{hyperref}

\theoremstyle{plain}
  \newtheorem{thm}{Theorem}[section]
  \newtheorem{defn}[thm]{Definition}
  \newtheorem{prop}[thm]{Proposition}
  
  \newtheorem{lem}[thm]{Lemma}
\theoremstyle{definition}
  \newtheorem{example}[thm]{Example}
  \newtheorem{rem}[thm]{Remark}

\newcommand{\on}{\operatorname}
\newcommand{\s}{\mathfrak{s}}
\newcommand{\g}{\mathfrak{g}}
\newcommand{\h}{\mathfrak{h}}

\renewcommand{\a}{\mathfrak{a}}
\renewcommand{\b}{\mathfrak b}

\newcommand{\la}{\langle}
\newcommand{\ra}{\rangle}
\newcommand{\gric}{\mathrm{GRic}}
\newcommand{\redu}{\mathrm{red}}
\newcommand{\pin}{\mathrm{Pin}}
\newcommand{\cl}{\on{Cl}}
\newcommand{\R}{\mathbb{R}}
\newcommand{\C}{\mathbb{C}}
\newcommand{\Z}{\mathbb{Z}}

\newcommand{\dv}{\on{div}}
\newcommand{\bw}{{\textstyle\bigwedge}}

\title[Courant algebroids, PL T-duality, and type II supergravities]{Courant algebroids, Poisson-Lie T-duality, and type II supergravities}

\thanks{Supported in part by  the grant MODFLAT of the European Research Council and the NCCR SwissMAP of the Swiss National Science Foundation.}
\author{Pavol \v{S}evera}
\address{Section of Mathematics, University of Geneva, Switzerland}
\email{pavol.severa@gmail.com}
\author{Fridrich Valach}
\address{Section of Mathematics, University of Geneva, Switzerland}
\email{fridrich.valach@gmail.com}

\begin{document}
\maketitle
\begin{abstract}
We reexamine the notions of generalized Ricci tensor and scalar curvature on a general Courant algebroid, reformulate them using objects natural w.r.t.\ pull-backs and reductions, and obtain them from the variation of a natural action functional. This allows us to prove, in a very general setup, the compatibility of the Poisson-Lie T-duality with the renormalization group flow and with string background equations. We thus extend the known results to a much wider class of dualities, including the cases with gauging (so called dressing cosets, or equivariant Poisson-Lie T-duality). As an illustration, we use the formalism to provide new classes of solutions of modified supergravity equations on symmetric spaces.
\end{abstract}

\tableofcontents

\section{Introduction}

Given a Riemannian metric $g$ and a closed 3-form $H$ on a manifold $M$ one can define a 2-dimensional $\sigma$-model with the target $M$. If $f:\Sigma\to M$ is a smooth map from a surface $\Sigma$ with a (pseudo)conformal structure, its action is
\begin{equation}\label{sigmaaction}
S(f)=\int_\Sigma g(\partial f,\bar\partial f)+\int_Yf^*H
\end{equation}
where $Y$ is a 3-fold with boundary $\Sigma$ and $f$ is extended to $Y$.

The pair $(g,H)$ is equivalent to a generalized metric in an exact Courant algebroid (CA). If in place of an exact CA we use an arbitrary CA $E$, a generalized metric in $E$ can still be used to define a 2-dim $\sigma$-model 
 \cite{SCS}. These more exotic 2-dim $\sigma$-models are the basis of \emph{Poisson-Lie T-duality}: if a suitable pull-back of the CA $E$ turns it to an exact CA, the exotic $\sigma$-model turns to  a  $\sigma$-model of the type \eqref{sigmaaction} for a certain $(M,g,H)$, and if there are several different suitable pullbacks of $E$ then all these $\sigma$-models are isomorphic as Hamiltonian systems (up to finitely many degrees of freedom).

If one adds the possibility of gauging, corresponding to reductions of equivariant CAs, the above phenomenon generalizes to \emph{equivariant Poisson-Lie T-duality} (introduced as ``dressing cosets'' in \cite{KS3}), significantly increasing the number of examples.

Not only $g$ and $H$, but also other massless fields appearing in string theory,
such as the dilaton and the Ramond-Ramond fields, can be conveniently described in terms of exact CAs \cite{CSCW} (a similar description works for the gauge field in the type I and heterotic case \cite{MGF2} where one needs certain transitive CAs). The string background equations and the corresponding action functionals look cleaner when seen through this perspective.

A natural question  is whether one can define a string effective action functional, dilaton, Ramond-Ramond fields, etc., for arbitrary (non-exact) CAs, in a way that would prove that Poisson-Lie T-duality is compatible with  string background equations, with Ricci flow (renormalization group flow), etc.

We solve this problem as follows. For any CA $E$ and any generalized metric $V_+\subset E$ we define a Laplacian $\Delta_{V_+}$ acting on half-densities, and consider the ``generalized string effective action functional'' $S_E(V_+,\sigma)=-\frac{1}{2}\int\sigma\Delta_{V_+}\sigma$ (where $\sigma$ is a half-density). The gradient flow of this functional (for a fixed $\sigma$) on the space of generalized metrics is the generalized Ricci flow, and the Euler-Lagrange equations of $S_E$ are the generalized string background equations, with $\sigma$ playing the role of the dilaton. Ramond-Ramond fields enter the picture as spinors and de Rham's $d$ is replaced with the generating Dirac operator from \cite{AX}.

For many purposes one can replace the half-density/dilaton by a divergence operator \cite{MGF}. Divergence operators, just as generalized metrics and spinors, are  well behaved under pull-backs and reductions of CAs, i.e.\ operations involved in  Poisson-Lie T-duality. Therefore, once we reformulate the string background equations or the Ricci flow in these terms, it is trivial to prove the above compatibility. 

Such results were only known in special cases \cite{VKS,SV,JV2} and required extensive calculations. The advantage of our approach is that it holds in the general case (including the equivariant Poisson-Lie T-duality) and at the same time gives simple proofs.


A part of the motivation for this work was to understand from the perspective of equivariant Poisson-Lie T-duality the $\eta$-deformed $AdS_5\times S^5$, found in \cite{DMV}, which is a solution of modified type IIB SUGRA equations \cite{AFHRT}. As an example of application of our formalism, we work out several new families of similar solutions, with one or more parameters.

An interesting open problem is to reconcile our approach via CAs (having predecessors e.g.\ in \cite{let,CG,CSCW,PLC,SCS,SV,JV,JV2}) with the Double Field Theory (DFT) approach of \cite{HZ,HKZ,H}. For the moment the mathematical content of DFT escapes us.

The article as structured as follows. We start in Section 2 by reviewing the notion of Courant algebroids, generalized metrics and divergences. Then, in Section 3 we present the generalized Ricci tensor and flow and in Section 4 the Laplacian $\Delta_{V_+}$, the generalized string effective action, and the generalized scalar curvature. Section 5 reviews the concept of Poisson-Lie T-duality, including also the cases of spectators and dressing cosets. It also contains proofs of the compatibility of the Poisson-Lie T-duality with string background equations in the cases of bosonic, type I, and heterotic string theory. Section 6 is devoted to the formalism needed to deal with the Ramond-Ramond fields. This allows us to extend the results of Section 5 to the type II string theory, which is done in Section 7. Finally, in Section 8, we construct the aforementioned families of solutions of modified supergravity equations of motion.

\section{Courant algebroids}
In this section we summarize some basic definitions and facts concerning Courant algebroids.

\subsection{General and exact Courant algebroids}
Courant algebroids, introduced in \cite{lxw}, are a generalization of Lie algebras with an invariant inner product. By definition, a \emph{Courant algebroid}  (CA) is a vector bundle $E\to M$ endowed with an non-degenerate symmetric bilinear form $\la\,,\,\ra$ on its fibres, with a vector bundle map $\rho:E\to TM$ called the \emph{anchor}, and with a $\R$-bilinear map $[\,,\,]:\Gamma(E)\times\Gamma(E)\to\Gamma(E)$ such that for all $u,v,w\in\Gamma(E)$ and $f\in C^\infty(M)$
\begin{align*}
[u,[v,w]]&=[[u,v],w]+[v,[u,w]]\\
[u,fv]&=f[u,v]+(u\cdot f)v\\
u\cdot\la v,w\ra&=\la[u,v],w\ra+\la v,[u,w]\ra\\
[u,v]+[v,u]&=d_E\la u,v\ra,
\end{align*}
where $u\cdot f:=\rho(u)f$ and $d_Ef\in\Gamma(E)$ is given by $\la u,d_Ef\ra=u\cdot f$.

A CA is \emph{exact} if the sequence
\begin{equation}\label{exact}
0\to T^*M\xrightarrow{\rho^t}E\xrightarrow{\rho}TM\to 0
\end{equation}
is exact. Exact CAs are classified by $H^3(M,\R)$: if we split the exact sequence \eqref{exact} so that $TM\subset TM\oplus T^*M\cong E$ is $\la\,,\,\ra$-isotropic then the 3-form $H\in\Omega^3(M)$ given by
$$H(u,v,w):=\la[u,v],w\ra\quad\forall u,v,w\in\Gamma(TM)\subset\Gamma(E)$$
is closed and its cohomology class is independent of the splitting. The Courant bracket $[,]$ on $E\cong TM\oplus T^*M$ is
$$[(u,\alpha),(v,\beta)]=\bigl([u,v],\mathcal L_u\beta-\iota_v d\alpha+H(u,v,\cdot)\bigr)\quad\forall u,v\in\Gamma(TM),\alpha,\beta\in\Gamma(T^*M)$$
and $\la\cdot,\cdot\ra$ is given by the natural pairing on $TM\oplus T^*M$.

\begin{example}\label{ex:exGH}
Let $\g$ be a Lie algebra with a non-degenerate invariant symmetric pairing $\la,\ra$ (i.e.\ $\g$ is a CA over a point), $\mathsf G$ a connected Lie group integrating $\g$, and $\mathsf H\subset\mathsf G$ a Lie subgroup such that $\h^\perp=\h$ (i.e.\ $\h\subset\g$ is a Lagrangian Lie subalgebra). Then the trivial vector bundle $\g\times\mathsf G/\mathsf H\to\mathsf G/\mathsf H$ is naturally an exact CA: the pairing and the bracket of constant sections are the same as in $\g$, and the anchor map $\rho$ is the action of $\g$ on $\mathsf G/\mathsf H$.

Exact CAs of this type play an important role in Poisson-Lie T-duality.
\end{example}

\subsection{Generalized (pseudo)metrics and divergences}
Let $E$ be a CA. We say that a subbundle $V_+\subset E$ is a \emph{generalized pseudometric} if $\la\cdot,\cdot\ra|_{V_+}$ is non-degenerate. We define $V_-:=V_+^\perp$.
In the special case when $\la\cdot,\cdot\ra|_{V_+}$ is positive-definite and $\la\cdot,\cdot\ra|_{V_-}$ is negative definite, we shall say $V_+$ is a \emph{generalized metric}.

Throughout the text we will identify $V_+, V_-$ (and also $E$) with their duals using $\la\cdot,\cdot\ra$. Furthemore, we will denote the orthogonal projections to $V_\pm$ by the subscripts $\pm$ and will also frequently write $a_+$ and $b_-$, instead of just $a$ and $b$, for sections (or elements) in $V_+$ and $V_-$, respectively. For example, $[a_+,b_-]_+$ denotes the projection of $[a_+,b_-]$ to $V_+$.

If $E$ is exact and $V_+\subset E$ is a generalized metric then  $\on{rank} V_+=\frac{1}{2}\on{rank} E$ and there exists a unique splitting $E\cong TM\oplus T^*M$ such that $V_+$ is the graph of a Riemannian metric $g$, i.e.
\[V_+=\{v +g(v,\cdot)\mid v\in \Gamma(TM)\}.\] Thus, a generalized metric in an exact CA is the same as a pair $(g,H)$, where $H$ is the closed 3-form given by the splitting. The data $(g,H)$ is what is needed to define a 2-dimensional $\sigma$-model with the target $M$.

\begin{example}\label{ex:V+GH}
Continuing Example \ref{ex:exGH}, if $V_+\subset\g$ is a generalized metric then the constant subbundle $V_+\times\mathsf G/\mathsf H$ of the exact CA $\g\times\mathsf G/\mathsf H$ is a generalized metric, and thus gives rise to a pair $(g,H)$ on $\mathsf G/\mathsf H$ and can be used to define a 2-dim $\sigma$-model with the target $\mathsf G/\mathsf H$. 

Poisson-Lie T-duality (in its simplest form) is the statement that this $\sigma$-model, as a Hamiltonian system, is (basically) independent of $\mathsf H$ and can be formulated in terms of $V_+\subset\g$ only.
\end{example}

More generally, an exact CA $E$ together with a generalized pseudometric $V_+$ transverse to $T^*M$ and satisfying $\on{rank} V_+=\frac{1}{2}\on{rank} E$ is equivalent to a pair $(g,H)$ of a pseudo-Riemannian metric and a closed 3-form.

Following \cite{AX,MGF}, a \emph{divergence} on $E$ is a  $\R$-linear map $\dv:\Gamma(E)\to C^\infty(M)$ such that 
\[
  \dv(fa)=f\dv a+a\cdot f\quad(\forall a\in\Gamma(E),f\in C^\infty(M)).
\]
Divergences form an affine space over $\Gamma(E)$: if $\dv$ and $\dv'$ are divergences then $\dv-\dv'=\la e,\cdot\ra$ for a section $e$ of $E$. 

If $\mu$ is an everywhere non-zero density on $M$ then $\dv_\mu a:=\mu^{-1}L_{\rho(a)}\mu$ is a divergence.

\begin{prop}\label{prop:dere}
If $\dv$ is a divergence on a CA $E\to M$ and $\mu$ a density on $M$, and if $e\in\Gamma(E)$ is given by $\dv-\dv_\mu=\la e,\cdot\ra$ then the derivation $[e,\cdot]$ of $\Gamma(E)$ is independent of $\mu$. For $a,b\in\Gamma(E)$ and $f\in C^\infty(M)$ we have
$$
\la [e,a],b\ra=\dv[a,b]-a\cdot\dv b+b\cdot\dv a,\quad e\cdot f=\dv d_Ef.
$$
In particular, the expression 
\begin{equation}\label{Zdiv}
\dv[a_+,b_-]-a_+\cdot\dv b_-+b_-\cdot\dv a_+,
\end{equation}
$a_+\in \Gamma(V_+)$ and $b_-\in\Gamma(V_-)$,
is $C^\infty(M)$-linear in both $a_+$ and $b_-$.
\end{prop}
\begin{proof}
Since
$$\dv_\mu[a,b]-a\cdot\dv_\mu b+b\cdot\dv_\mu a =0\quad \forall a,b\in\Gamma(E),$$
 we have
\begin{multline*}
\dv[a,b]-a\cdot\dv b+b\cdot\dv a =
  \la e,[a,b]\ra-a\cdot \la e,b\ra + b \cdot \la e,a\ra\\
=-\la[a,e],b\ra + b \cdot \la e,a\ra
  =\la [e,a],b\ra.
\end{multline*}
As $\la[e,a],b\ra$ is independent of $\mu$ for every $a,b\in\Gamma(E)$, so is $[e,\cdot]$.
The $C^\infty$-bilinearity of $\la[e,a_+],b_-\ra$ is easily checked. As $\rho(d_Ef)=0$, we have $\dv_\mu d_Ef=0$ and thus $e\cdot f=\la e,d_Ef\ra=\dv d_Ef-\dv_\mu d_Ef=\dv d_Ef$.
\end{proof}

We shall say that $\dv$ is \emph{compatible with $V_+$}  if $V_+$ is invariant under the derivation $[e,\cdot]$ (i.e.\ if $[e,a_+]\in\Gamma(V_+)$ $\forall a_+\in\Gamma(V_+)$), i.e.\ if
\begin{align}
\label{comp_div}
\dv[a_+,b_-]-a_+\cdot\dv b_-+b_-\cdot\dv a_+=0.
\end{align}
This is always true if $\dv=\dv_\mu$ for a density $\mu$ (as then $[e,\cdot]=0$).

\begin{example}\label{ex:gendil}
Suppose that $E\to M$ is exact and $V_+\subset E$ a generalized metric, corresponding to a pair $(g,H)$. As noticed above, $V_+$ gives us a splitting $E\cong TM\oplus T^*M$. Let $\mu_g$ be the density on $M$ given by $g$. Any divergence on $E$ is of the form $\dv=\dv_{\mu_g}+\la(X,\alpha),\,\cdot\,\ra$ for some $(X,\alpha)\in\Gamma(TM\oplus T^*M)$, and $\dv$ is compatible with $V_+$ iff 
$$\mathcal L_Xg=0\text{ and }d\alpha=i_XH.$$
 Such pairs $(X,\alpha)$ serve as a replacement of dilaton in modified SUGRA \cite{TW,AFHRT}. An actual dilaton $\phi$ exists if $\dv=\dv_{e^{-2\phi}\mu_g}$. In this case $X=0$ and $\alpha=-2\,d\phi$ is exact. For the existence of a dilaton it is thus enough to know that $X=0$ and $H^1(M,\R)=0$. (The idea of seeing divergencies on exact CAs as a generalization of dilaton was put forward in \cite{MGF}.)
\end{example}

\begin{example}\label{ex:divGH}
Continuing Examples \ref{ex:exGH} and \ref{ex:V+GH}, there is a natural divergence $\dv$ on the exact CA $\g\times\mathsf G/\mathsf H$, given by $\dv u=0$ for every $u\in\g$ (constant section of the CA). This divergence is compatible with the generalized metrics of the form $V_+\times\mathsf G/\mathsf H$ ($V_+\subset\g$) as \eqref{comp_div} is trivially satisfied for constant sections $a_+,b_-$.

The divergence $\dv$ is of the form $\dv_\mu$ iff $\mu$ is a $\g$-invariant density on $\mathsf G/\mathsf H$ (as $\dv u=0$ $\forall u\in\g$) and such a $\mu$ exists iff $\h$ is a unimodular Lie algebra ($\g$ is automatically unimodular due to its invariant pairing $\la,\ra$). (This explains why in Poisson-Lie T-duality, the dilaton exists only in the case of unimodular $\h$'s, and is otherwise replaced by a generalized dilaton of Example \ref{ex:gendil}.)

The derivation $[e,\cdot]$, see Proposition \ref{prop:dere}, satisfies $[e,a]=0$ for every constant section $a\in\g$, i.e.\ it is completely determined by the vector field $X:=\rho(e)$, and the formula $e\cdot f=\dv d_E f$ gives us
$$X=\rho(e^\alpha)\rho(e_\alpha)$$
where $e_\alpha$ is a basis of $\g$ and $e^\alpha$ is the dual basis of $\g$ ($X$ is written as a 2nd order differential operator, but is in fact a $\g$-invariant vector field on $\mathsf G/\mathsf H$).
\end{example}

\section{Generalized Ricci tensor and generalized Ricci flow}

Any generalized (pseudo)metric $V_+\subset E$ comes with a natural infinitesimal deformation of $V_+$ in $E$, defined up to inner derivations of $E$ \cite{SV,MGF}. In this section we give a new definition of this deformation and prove its basic properties.

\begin{defn}
The \emph{generalized Ricci tensor}  $\gric_{V_+,\dv}$, corresponding to a pair $(V_+,\dv)$, is the map $\on{GRic}_{V_+,\dv}:\Gamma(V_+)\times\Gamma(V_-)\to C^\infty(M)$  given by

\begin{equation}
\label{gricci}
\on{GRic}_{V_+,\dv}(a_+,b_-):=\dv[b_-,a_+]_+ - b_-\cdot\dv a_+ -
\on{Tr}_{V_+}[[\,\cdot\,,b_-]_-,a_+]_+
\end{equation}
\end{defn}
See \cite{MG,CSCW,MGF2,JV,SV,MGF} for equivalent definitions using auxiliary connections; the fact that $\gric$ really depends on a pair $(V_+,\dv)$ was observed in \cite{MGF}.
If the CA $E$ is not clear from the context, we shall use the notation $\on{GRic}^E_{V_+,\dv}$.
In the definition we use the fact that $[\,\cdot\,,b_-]_-:\Gamma(V_+)\to\Gamma(V_-)$ and $[\,\cdot\,,a_+]_+:\Gamma(V_-)\to\Gamma(V_+)$ are $C^\infty(M)$-linear, so $\on{Tr}_{V_+}[[\,\cdot\,,b_-]_-,a_+]_+$ is well-defined. A motivation for this definition can be found below, in Equation \eqref{tang}.

\begin{prop}
$\on{GRic}_{V_+,\dv}$ is $C^\infty(M)$-bilinear, i.e.\ it is a section of $V_+^*\otimes V_-^*$.
\end{prop}
\begin{proof}
We have
$$\on{Tr}_{V_+}[[\,\cdot\,,fb_-]_-,a_+]_+=f\on{Tr}_{V_+}[[\,\cdot\,,b_-]_-,a_+]_+ +[b_-,a_+]_+\cdot f.$$

On the other hand, 
$$\dv [fb_-,a_+]_+=f\dv[b_-,a_+]_+ + [b_-,a_+]_+\cdot f$$ and thus GRic is $C^\infty(M)$-linear in $b_-$. The $C^\infty(M)$-linearity in $a_+$ then follows from 
\begin{multline}\label{gricflip}
\gric_{V_-,\dv}(b_-,a_+)=\gric_{V_+,\dv}(a_+,b_-)\\ +\dv[a_+,b_-]-a_+\cdot\dv b_-+b_-\cdot\dv a_+
\end{multline}
 and from the  $C^\infty(M)$-bilinearity of \eqref{Zdiv}.
\end{proof}

Importantly, in the case of an exact CA, we recover the usual Ricci tensor:

\begin{prop}\label{exactRicci}
Suppose $E,V_+$ is an exact CA with a generalized (pseudo)metric corresponding to a pair $(g,H)$. Let $\on{Ric}_{g,H}$ be the Ricci tensor of $\nabla=\nabla_{g,H}$, the $g$-preserving connection with the torsion given by $H$, and let $\dv =\dv_{\mu_g}=\on{Tr} \nabla\circ \rho$ where $\mu_g$ is the density given by $g$. Then
\[\on{GRic}_{V_+,\dv}(a_+,b_-)=\on{Ric}_{g,H}(\rho(a_+),\rho(b_-)).\]
\end{prop}
\begin{proof}
The key point is the following observation:
\[\rho([b_-,a_+]_+)=\nabla_{\rho(b_-)}\rho(a_+),\quad \rho([a_+,b_-]_-)=\tilde\nabla_{\rho(a_+)}\rho(b_-),\]
where $\tilde\nabla$ is the connection with the opposite torsion, i.e. $\tilde\nabla_X Y=\nabla_Y X + [X,Y]$. For this fact, see \cite{MG}. The proposition then follows from the identity
\begin{align}
\label{tang}
\on{Ric}_\nabla(X,Y)=\dv_\nabla\nabla_Y X-Y\dv_\nabla X-\on{Tr}(\nabla X\,\tilde\nabla Y),
\end{align}
valid for the Ricci tensor of any  connection $\nabla$ on $TM$, where $\dv_\nabla X:=\on{Tr}\nabla X$.
\end{proof}

Let us now show how $\gric$ transforms when we change the divergence.
\begin{prop}
If $\dv'={\dv}+\la e,\cdot\ra$ then 
$$\on{GRic}_{V_+,\dv'}(a_+,b_-)-\on{GRic}_{V_+,\dv}(a_+,b_-)=-\la[e_+,a_+],b_-\ra$$
\end{prop}
\begin{proof}
We have
$$ 
\la e,[b_-,a_+]_+\ra -b_-\cdot\la e,a_+\ra=-\la[b_-,e_+]_+,a_+ \ra
=\la[e_+,b_-],a_+\ra=-\la[e_+,a_+],b_-\ra.
$$
\end{proof}

Any linear map $V_+\to V_-$, in particular $\gric_{V_+,\dv}$ (understood as a section of $V_+^*\otimes V_-$), can be seen as an infinitesimal deformation of $V_+$ in $E$. The flow of $V_+$ given by $\gric_{V_+,\dv}$ is the \emph{generalized Ricci flow} (we do not attempt to give conditions for short-time existence of this flow). The previous proposition says that, up to infinitesimal automorphisms of $E$, this flow is independent of the choice of $\dv$: the difference of the flows given by $\dv'$ and by $\dv$ is the inner derivation $-[e_+,\,\cdot\,]$ of $E$.

In the case of an exact CA the generalized Ricci flow is, by Proposition \ref{exactRicci}, the 1-loop renormalization group flow of the 2-dimensional $\sigma$-model given by the pair $(g,H)$; for $H=0$ it is the usual Ricci flow of $g$.

\begin{example}\label{ex:ricGH}
Continuing Example \ref{ex:divGH}, we have
\begin{multline*}
\gric^{\g\times\mathsf G/\mathsf H}_{V_+\times\mathsf G/\mathsf H,\dv}(a_+,b_-)=\gric^\g_{V_+,0}(a_+,b_-)\\
=-\on{Tr}_{V_+}[[\,\cdot\,,b_-]_-,a_+]_+ \quad(\forall a_+\in V_+,b_-\in V_-)
\end{multline*}
In particular, the generalized Ricci flow of $V_+\times\mathsf G/\mathsf H$ in the exact CA $\g\times\mathsf G/\mathsf H$ is simply equal (pointwise in $\mathsf G/\mathsf H$) to the generalized Ricci flow of $V_+\subset\g$ (which implies a short-time existence of the flow, as the latter is an ODE). In other words, Poisson-Lie T-duality (in the case of no spectators) is compatible with the 1-loop renormalization group flow.
\end{example}

\begin{rem}
If $\dv=\dv_\mu+\la e,\cdot\ra$ for a density $\mu$ and a section $e\in\Gamma(E)$, we have, by \eqref{gricflip} and Proposition \ref{prop:dere},
$$\gric_{V_-,\dv}(b_-,a_+)=\gric_{V_+,\dv}(a_+,b_-)+\la[e,a_+],b_-\ra$$
In particular, if $\dv$ is compatible with $V_+$ (i.e.\ if $[e,\cdot]$ preserves $V_+\subset E$) then $\gric_{V_-,\dv}(b_-,a_+)=\gric_{V_+,\dv}(a_+,b_-)$.

A natural question is, if $\dv$ is compatible with $V_+$,  whether it stays compatible during the generalized Ricci flow. This is equivalent to $[e,\cdot]$-invariance of $\gric_{V_+,\dv}$, i.e.\ to
$$e\cdot\gric_{V_+,\dv}(a_+,b_-)=\gric_{V_+,\dv}([e,a_+],b_-)+\gric_{V_+,\dv}(a_+,[e,b_-]).$$
As $V_+$ is assumed to be $[e,\cdot]$-invariant, this equation is satisfied provided $\dv$ is $[e,\cdot]$-invariant.
As
$$e\cdot\dv a-\dv[e,a]=a\cdot\bigl(\dv_\mu e+\tfrac{1}{2}\la e,e\ra\bigr),$$
a sufficient condition is that the function $\dv_\mu e+\tfrac{1}{2}\la e,e\ra$ (which depends only on $\dv$ and not on the choice of its splitting $\dv=\dv_\mu+\la e,\cdot\ra$) is constant along the integral leaves of $E$.
\end{rem}

\section{Laplacian, string effective action, and generalized scalar curvature}
Given a generalized (pseudo)metric $V_+\subset E$ we define in this section a natural 2nd-order differential operator (Laplacian) $\Delta_{V_+}$ acting on half-densities. The functional $S_E(V_+,\sigma):=-\frac{1}{2}\int_M \sigma\Delta_{V_+}\sigma$, where $\sigma$ is a half-density, is a generalization of the low energy string effective action; its gradient flow (for a fixed $\sigma$) is the generalized Ricci flow.

\subsection{Computing with a local frame}
Suppose that $E\to M$ is a CA and that $e_\alpha$ is a local basis of $E$ such that $\la e_\alpha,e_\beta\ra$ are constant functions. 
It is easy to see that the expression $c_{\alpha\beta\gamma}:=\la e_\alpha, [e_\beta,e_\gamma]\ra$ is totally antisymmetric. If we change the frame $e_\alpha$  by an infinitesimal orthogonal transformation $\delta e_\alpha=A^\beta_{\,\,\alpha} e_\beta$, for $A_{\alpha\beta}$ skew-symmetric (indices are lowered and raised using $\la,\ra$), we obtain (see \cite{DR}, formulas 3.3 and 4.7)
\begin{equation}\label{c_transform}
\delta c_{\alpha \beta \gamma}=3A^\delta_{\,\,[\alpha}c_{\beta\gamma]\delta}-3e_{[\alpha}\cdot A_{\beta \gamma]}.
\end{equation}

If $V_+\subset E$ is a generalized (pseudo)metric, we shall furthermore suppose that the (local) basis $e_\alpha$ is the union of a basis $e_a$ of $V_+$ and of a basis $e_{\bar a}$ of $V_-$ (in general, $\alpha,\beta,\dots$ will correspond to the basis of $E$ and $a,b,\dots$ and $\bar a,\bar b,\dots$ to the bases of $V_+$ and of $V_-$ respectively).

\subsection{Laplacian}

Given a generalized (pseudo)metric $V_+\subset E$ there is a natural 2nd-order formally self-adjoint operator $\Delta_{V_+}$ acting on half-densities on $M$. To define it we use local bases $e_a$ and $e_{\bar a}$  of $V_+$ and $V_-$ as above.
\begin{defn} The Laplacian given by a generalized (pseudo)metric $V_+\subset E$ is the differential operator acting on half-densities on $M$
$$
\Delta_{V_+}:=4\,\mathcal{L}_{\rho(e^a)}\mathcal{L}_{\rho(e_a)}-\frac{1}{6}c_{abc}\,c^{abc}-\frac{1}{2}c_{ab\bar c}\,c^{ab\bar c}
$$
\end{defn}

\begin{prop}
$\Delta_{V_+}$ is well defined, i.e.\ it does not depend on the choice of the local trivialization of $E$.
\end{prop}
\begin{proof}
Suppose we make an infinitesimal change of the local basis $e_a$ of $V_+$ given by $\delta e_a=A^b_{\,\,a} e_b$ for $A_{ab}$ skew-symmetric. (The operator $\Delta_{V_+}$ is clearly independent of the choice of the basis $e_{\bar a}$.) Let us examine the transformation properties of the respective terms in the definition of $\Delta_{V_+}$.

A straightforward calculation (using the fact that when acting on half-densities, one has $\mathcal{L}_{fu}=f\mathcal{L}_u+\frac{1}{2}\mathcal{L}_u f$) gives
\[
\delta (4\mathcal{L}_{\rho(e^a)}\mathcal{L}_{\rho(e_a)}\sigma)=2(\mathcal{L}_{\rho(e_a)}\mathcal{L}_{\rho(e_b)}A^{ba})\sigma=([e_a,e_b]\cdot A^{ba})\sigma.
\]

For the rest, let us extend $A$ from $\bigwedge^2 V_+$ to $\bigwedge^2 E$ by setting $A_{\bar{a}b}=A_{a\bar{b}}=A_{\bar{a}\bar{b}}=0$. Using (\ref{c_transform}) we have
\begin{align*}
-\frac{1}{6}\delta (c_{abc}c^{abc})&=-(e_c\cdot A_{ab})c^{abc}=-[e_a,e_b]_+\cdot A^{ab},\\
-\frac{1}{2}\delta (c_{ab\bar c}c^{ab\bar c})&=-(e_{\bar c}\cdot A_{ab})c^{ab\bar c}=-[e_a,e_b]_-\cdot A^{ab},
\end{align*}
which combines with $\delta (\mathcal{L}_{\rho(e^a)}\mathcal{L}_{\rho(e_a)}\sigma)$ to give 0.
\end{proof}

\begin{rem}
When $V_+=E$, the operator $\Delta_{V_+}$ is actually a function, namely 8 times the square of the canonical generating Dirac operator. See Remark \ref{rem:D^2}. 
\end{rem}

Consider a pseudo-Riemannian manifold $(M,g)$ with a closed 3-form $H$. Let $E$, $V_+$ be the corresponding exact CA with generalized pseudometric. 
Let  $R$ be the scalar curvature of $g$ and $\mu_g$  the density on $M$ given by $g$. In this case we can express $\Delta_{V_+}$ in terms of familiar quantities:
\begin{prop}\label{prop:Deltaex}
For any $f\in C^\infty(M)$ we have 
$$\Delta_{V_+}(f\mu_g^{1/2})=\mu_g^{1/2}(2\Delta_g-\tfrac{1}{2}R+\tfrac{1}{4}H^2) f,$$
where $H^2=\frac{1}{3!}H_{abc}H^{abc}$ and $\Delta_g$ is the usual Laplace operator acting on functions on $M$.
\end{prop}
\begin{proof}
For a fixed point $p\in M$ choose a local normal basis of vector fields $E_a$ such that $g(E_a,E_b)$ are constant functions (normality means $\Gamma^a_{\,\,\,bc}=0$ at $p$, where the Christoffel's symbols $\Gamma^a_{\,\,\,bc}$ are defined by $\nabla_{E_c}E_b=\Gamma^a_{\,\,\,bc}E_a$ and $\nabla$ is the Levi-Civita connection; in particular $[E_a,E_b]=0$ at $p$). Then $e_a=E_a+g(E_a,\cdot)\in\Gamma(E)$ form a basis of $V_+$ and $\la e_a,e_b\ra=2g(E_a,E_b)$. If we denote by $E^a$ the dual basis (of $TM$) to $E_a$, then $e^a=\textstyle{\frac{1}{2}}(E^a+g(E^a,\cdot))$.

Observe that 
$$\mathcal L_{\rho(e_a)}\mu_g^{1/2}=\mathcal L_{E_a}\mu_g^{1/2}=\tfrac{1}{2}\mu_g^{1/2}\on{tr}\nabla E_a=\tfrac{1}{2}\Gamma^b_{\,\,\,ab}\mu_g^{1/2}$$
hence at $p$ (using $\rho(e^a)=\frac{1}{2}E^a$ and  $\Gamma^a_{\,\,\,bc}|_p=0$)
$$4\mathcal L_{\rho(e^a)}\mathcal L_{\rho(e_a)}\mu_g^{1/2}=
2\mathcal L_{E^a}\mathcal L_{E_a}\mu_g^{1/2}=
\Gamma^{b\hphantom{ab,}a}_{\hphantom{b}ab,}\mu_g^{1/2}.$$

A quick glance reveals that at the point $p$ one has $[e_a,e_b]\in T^* M\subset E$ as well as $\langle [e_a,e_b],e_c\rangle=H_{abc}$ and thus $\langle [e^a,e^b],e^c\rangle=\textstyle{\frac{1}{8}}H^{abc}$, which implies
\[c_{ab\gamma}c^{ab\gamma}=\langle [e_a,e_b],[e^a,e^b]\rangle=0,\quad c_{abc}c^{abc}= \langle [e_a,e_b],e_c\rangle\langle [e^a,e^b],e^c\rangle=\tfrac{1}{8}H_{abc}H^{abc}.\] 
and since
$$-\tfrac{1}{6}c_{abc}\,c^{abc}-\tfrac{1}{2}c_{ab\bar c}\,c^{ab\bar c}=
\tfrac{1}{3}c_{abc}\,c^{abc}-\tfrac{1}{2}c_{ab\gamma}\,c^{ab\gamma},
$$
we get
$$\mu_g^{-1/2}\Delta_{V_+}\mu_g^{1/2}=\Gamma^{b\,\,\,\,\,\,\,a}_{\,\,\,ab,}+\tfrac{1}{24}H_{abc}H^{abc}=-\tfrac{1}{2}R+\tfrac{1}{4}H^2.$$

To finish the proof we simply notice that
\begin{multline*}
\mathcal L_{\rho(e_a)}\mathcal L_{\rho(e^a)}(f\sigma)=f\mathcal L_{\rho(e_a)}\mathcal L_{\rho(e^a)}\sigma+2\mathcal L_{(e^a\cdot f)\,\rho(e_a)}\sigma\\
=f\mathcal L_{\rho(e_a)}\mathcal L_{\rho(e^a)}\sigma+\dv_{\sigma^2}(d_+f)\,\sigma
\end{multline*}
where $d_+f=(e^a\cdot f)\,e_a\in\Gamma(V_+)$.
In our case $\rho(d_+f)=\frac{1}{2}\on{grad}_g f$, so
$$\mathcal L_{\rho(e_a)}\mathcal L_{\rho(e^a)}(f\mu_g^{1/2})=f\mathcal L_{\rho(e_a)}
\mathcal L_{\rho(e^a)}\mu_g^{1/2} +\tfrac{1}{2}(\Delta_g f)\mu_g^{1/2}$$
which concludes the proof.
\end{proof}

\subsection{String effective action}
If $\sigma$ denotes a half-density on $M$ and $V_+$ a generalized metric in a CA $E\to M$, let us define the action functional
$$S_E(V_+,\sigma)=-\frac{1}{2}\int_M \sigma \Delta_{V_+}\sigma.$$

\begin{example}\label{ex:SEF}
If $M$ is compact and $E$ is exact, so that $V_+\subset E$ is equivalent to a pair $(g,H)$, and if $\sigma=e^{-\phi}\mu_g^{1/2}$ for a suitable function $\phi\in C^\infty(M)$ (i.e.\ if $\sigma$ is everywhere positive), then (see Proposition \ref{prop:Deltaex})
$$S_E(V_+,\sigma)=\int_M \bigl(\tfrac{1}{4}R-\tfrac{1}{8}H^2 + \Vert d\phi\Vert_g^2\bigr)\,e^{-2\phi}\mu_g$$
is the low energy string effective action of the metric $g$, closed 3-form $H$, and dilaton $\phi$. 
\end{example}
 
We now examine (for a general CA) the Euler-Lagrange equations coming from the functional $S_E(V_+,\sigma)$. If $\sigma$ is a nowhere-vanishing half-density, we shall use the notation
$$\gric_{V_+,\sigma}:=\gric_{V_+,\dv_{\sigma^2}}.$$
\begin{thm}\label{thm:EOM1}
Under an infinitesimal change of $V_+$ given by $\varphi\in \Gamma(V_+\otimes V_-)$, for $\sigma$ fixed and nowhere vanishing, 
$$\delta_\varphi S_E(V_+,\sigma)=\int_M\on{GRic}_{V_+,\sigma}(\varphi)\,\sigma^2.$$
\end{thm}
\begin{proof}
We can suppose that the support of $\varphi$ is contained in a region where we chose
bases $e_a$, $e_{\bar a}$ of $V_+$ and $V_-$. If $\varphi=\varphi_{\bar ab}\,e^b\otimes e^{\bar a}$, let us extend $\varphi_{\bar ab}$ to a skew-symmetric matrix  via $\varphi_{ab}=\varphi_{\bar a \bar b}=0$, $\varphi_{a\bar b}=-\varphi_{\bar b a}$. Then $\delta_\varphi e_a=\varphi^{\bar b}_{\,\,a}e_{\bar b}$ and $\delta_\varphi e_{\bar a}=\varphi^{b}_{\,\,\bar a}e_{ b}$ give us bases of the deformed $V_+$ and $V_-$ and using (\ref{c_transform}) we get

\begin{align*}
\frac{1}{12}\delta_\varphi(c_{abc}c^{abc})&=\frac{1}{2}\varphi^{\bar d}_{\,\,a}c_{bc\bar d} c^{abc},\\
\frac{1}{4}\delta_\varphi(c_{ab\bar c}c^{ab \bar c})&=-\frac{1}{2}\varphi^{\bar d}_{\,\,a}c_{bc\bar d} c^{abc}+\varphi^{\bar d}_{\,\,a}c_{b\bar c\bar d} c^{ab\bar c}-(e_{a}\cdot \varphi_{b\bar c})c^{ab\bar c}.\\
\end{align*}
Their sum is
\begin{align*}
\varphi^{\bar d}_{\,\,a}c_{b\bar c\bar d} c^{ab\bar c}-(e_{a}\cdot \varphi_{b\bar c})c^{ab\bar c}&=-\on{Tr}_{V_+}\bigl([[\,\cdot\,,e_{\bar d}]_-,e_a]_+\bigr)\,\varphi^{\bar d a}-[e_{\bar d},e_a]_+\cdot \varphi^{\bar d a}\\
&=-\on{Tr}_{V_+}[[\,\cdot\,,\varphi^{\bar d a}e_{\bar d}]_-,e_a]_+,
\end{align*}
corresponding to the last term of (\ref{gricci}).

Integration by parts gives us
\begin{align*}
\delta_\varphi(-2\int_M \sigma\mathcal{L}_{\rho(e_a)}\mathcal{L}_{\rho(e^a)}\sigma)&=2\delta_\varphi\!\!\int_M (\mathcal{L}_{\rho(e_a)}\sigma) (\mathcal{L}_{\rho(e^a)}\sigma)=4\int_M (\mathcal{L}_{\rho(\varphi^{\bar b}_{\,\,a} e_{\bar b})}\sigma) (\mathcal{L}_{\rho(e^a)}\sigma)\\
&=\int_M \sigma^2 \dv_{\sigma^2} (\varphi^{\bar b}_{\,\,a} e_{\bar b}) \dv_{\sigma^2} e_a=-\int_M \sigma^2\; \varphi^{\bar b}_{\,\,a} e_{\bar b}\cdot \dv_\sigma e_a,
\end{align*}
corresponding to the second term in (\ref{gricci}).

Finally, note that $\int_M(\dv_{\sigma^2} s)\,\sigma^2=\int_M\mathcal L_{\rho(s)}\sigma^2=0$ for any (compactly supported) $s\in\Gamma(E)$, so the first term of \eqref{gricci}  does not contribute to the integral.
\end{proof}

\begin{rem}
If $M$ is compact and $\sigma$ nowhere-vanishing then there is a natural Riemannian metric on the (infinite-dimensional) space $\mathcal G_E$ of generalized metrics in $E$, given by (for $\varphi_1,\varphi_2\in\Gamma(V_+\otimes V_-)$ two tangent vectors at a generalized metric $V_+\subset E$)
$$g_\sigma(\varphi_1,\varphi_2)=-\int_M\la\varphi_1,\varphi_2\ra\,\sigma^2$$
Theorem \ref{thm:EOM1} now says that $\gric_{V_+,\sigma}$, seen as a vector field on $\mathcal G_E$, satisfies
$$\gric_\sigma=-\on{grad}_{g_\sigma} S_E$$
i.e.\ that the generalized Ricci flow is a gradient flow.
\end{rem}

It is much easier to see how the action transforms under a variation of the half-density, $\delta \sigma=\varsigma$:
\[\delta_\varsigma S_E=-\int_M \varsigma \Delta_{V_+} \sigma.\]
The Euler-Lagrange equations given by the action functional $S_E$ are thus (in the case of nowhere-vanishing $\sigma$)
\begin{subequations}\label{ELsugra}
\begin{align}
\gric_{V_+,\sigma}&=0\\
\Delta_{V_+}\sigma&=0.
\end{align}
\end{subequations}
In the case of an exact CA (see  Example \ref{ex:SEF}) these are the string background equations.

\subsection{Generalized scalar curvature and generalized string background equations}

Given a generalized metric $V_+\subset E$ and a non-vanishing half-density, we can define the ``generalized scalar curvature'' by 
$$\mathcal R_{V_+,\sigma}:=\sigma^{-1}\Delta_{V_+}\sigma\in C^\infty(M).$$
It is useful to extend this definition to cases when in place of $\sigma$ we have just a divergence $\dv$ on $E$.

\begin{defn}
Let $V_+\subset E$ be a generalized (pseudo)metric in an arbitrary CA $E$, and let $\dv$ be a divergence on $E$. Then the \emph{generalized scalar curvature} $\mathcal R_{V_+,\dv}$ is defined by
$$\mathcal R_{V_+,\dv}=\sigma^{-1}\Delta_{V_+}\sigma+\la e_+,e_+\ra + 2\dv_{\sigma^2}e_+$$
where $\sigma$ is an arbitrary non-vanishing half-density on $M$ and $e\in\Gamma(E)$ is given by 
$$\dv-\dv_{\sigma^2}=\la e,\cdot\ra.$$
\end{defn}
 An easy calculation gives 
\begin{equation}\label{genR}
\mathcal R_{V_+,\dv}=(\dv e^a)(\dv e_a)+2e^a\cdot(\dv e_a)-\frac{1}{6}c_{abc}\,c^{abc}-\frac{1}{2}c_{ab\bar c}\,c^{ab\bar c}
\end{equation}
which also implies that $\mathcal R_{V_+,\dv}$ doesn't depend on the choice of $\sigma$.
$\mathcal R_{V_+,\dv}$ was previously introduced (using different means) in \cite{JV} and $\mathcal R_{V_+,\sigma}$ (in the case of exact CAs) in \cite{CSCW}.

\begin{defn}
The \emph{generalized string background equations} for $V_+$ and $\dv$ are
\begin{align*}
\gric_{V_+,\dv}&=0\\
\mathcal R_{V_+,\dv}&=0
\end{align*}
\end{defn}

In special cases these equations appear in string theory:

\begin{itemize}
\item If $E$ is exact and $\dv$ is given by a half-density $\sigma$ then they are the string background equations for the corresponding triple $(g,H,\phi)$, see Example \ref{ex:SEF}

\item As discovered in \cite{MGF2}, if $E$ is transitive,  $\rho|_{V_+}$ is bijective, and $\dv$ is given by a half-density $\sigma$, then they are the string background equations for the type I or heterotic superstring, i.e.\ the type I SUGRA equations.

\item If $\dv$ is not given by a half-density, but is compatible with $V_+$, then they are the modified SUGRA equations of \cite{TW,AFHRT}.
\end{itemize}

\begin{example}\label{ex:sugraGH}
In the setup of Example \ref{ex:divGH} we have
\begin{align*}
\gric^{\g\times\mathsf G/\mathsf H}_{V_+\times\mathsf G/\mathsf H,\dv}&=\gric^{\g}_{V_+,0}\qquad\text{(Example \ref{ex:ricGH})}\\
\mathcal R^{\g\times\mathsf G/\mathsf H}_{V_+\times\mathsf G/\mathsf H,\dv}&=\mathcal R^\g_{V_+,0}=-\frac{1}{6}c_{abc}\,c^{abc}-\frac{1}{2}c_{ab\bar c}\,c^{ab\bar c}
\end{align*}
As a result, the generalized metric (or the corresponding pair $(g,H)$) $V_+\times\mathsf G/\mathsf H\subset \g\times\mathsf G/\mathsf H$ and the divergence $\dv$ (or the corresponding pair $(X,\alpha)$ - see Example \ref{ex:gendil}) form a solution of the generalized string background equations iff the algebraic equations
\begin{align*}
\gric^\g_{V_+,0}&=0\\
\mathcal R^\g_{V_+,0}&=0
\end{align*}
are satisfied. These equations do not depend on $\h$ - this fact is the Poisson-Lie T-duality for generalized string background equations. Dilaton exists if $\dv=\dv_\mu$ for a (necessarily $\g$-invariant) density on $\mathsf G/\mathsf H$, i.e.\ if $\h$ is unimodular. In that case, as $\sigma=\mu^{1/2}=e^{-\phi}\mu_g^{1/2}$, we get the dilaton field
$$\phi=\frac{1}{2}\log\frac{\mu_g}{\mu}.$$

\end{example}

\section{Poisson-Lie T-duality}

Poisson-Lie T-duality \cite{KS} is a generalization of the usual T-duality from torus bundles to more general manifolds with possibly no isometries. In this section we shall recall its formulation in terms of CAs \cite{let, PLC,SCS,SV} and prove its compatibility with the renormalization group flow and with the (generalized) string background equations.

 Some of these results are already in the literature, either without using CAs \cite{VKS,SS,SST} (for the Ricci flow) or with using CAs \cite{JV2} (for the string background equations), but they are restricted to the case of no spectators and require extensive calculations.  Our approach gives simple and transparent proofs in full generality (for the problem of the Ricci flow it follows our previous proof \cite{SV}, but gives a stronger result) and the problem of determining the dilaton becomes basically trivial. We also discuss the case of ``dressing cosets'' \cite{KS3} (equivariant PL T-duality) which is necessary for  most of the interesting examples (see Section \ref{sec:ex}).

\subsection{PL T-duality without spectators}

Let us first describe PL T-duality in the simple case of ``no spectators''. It was discussed in the previous text in several Examples, but let us summarize it here and add some more details. 

The main idea is, see Example \ref{ex:exGH}, that given a Lie algebra $\g$ with an invariant $\la,\ra$ (i.e.\ given a CA over a point), we have an \emph{exact} CA  $\g\times \mathsf G/\mathsf H$ whenever $\h^\perp=\h$. If $V_+\subset\g$ is a generalized metric then we  use the generalized metric $V_+\times \mathsf G/\mathsf H$ in the exact CA $\g\times \mathsf G/\mathsf H$ (Example \ref{ex:V+GH}) and, finally, we define a divergence by setting $\dv u=0$ for every \emph{constant} section $u\in\g$ (Example \ref{ex:divGH}).

Poisson-Lie T-duality is then the statement that various physically relevant properties of this generalized metric and divergence can be obtained directly from $V_+\subset\g$, i.e.\ that these properties are independent of $\h$.

\begin{rem}
Both ``Poisson-Lie'' and ``duality'' in ``PL T-duality'' come from the case when there are two Lagrangian Lie subalgebras $\h,\h^*\subset\g$ such that $\h\cap\h^*=0$. ``Duality'' then refers to the two manifolds $\mathsf G/\mathsf H$ and $\mathsf G/\mathsf H^*$, and ``Poisson-Lie'' to the fact that $\mathsf H$ and $\mathsf H^*$ are Poisson-Lie groups.
\end{rem}

\subsubsection{Sigma model}
The generalized metric $V_+\times \mathsf G/\mathsf H$ in the exact CA $\g\times \mathsf G/\mathsf H$ is equivalent to a pair $(g,H)$ on $\mathsf G/\mathsf H$ and thus gives rise to a 2-dim $\sigma$-model with the target $\mathsf G/\mathsf H$. This $\sigma$-model, as a Hamiltonian system, can be (up to finitely many degrees of freedom) formulated in terms of $V_+\subset\g$:

The phase space is the moduli space of flat $\g$-connection $A\in\Omega^1(Z)\otimes\g$ on an annulus $Z$, with the boundary condition $A|_{\text{inner }S^1}\in\Omega^1(S^1)\otimes\h$, modulo gauge transformations vanishing on the outer boundary circle and taking values in $\mathsf H$ on the inner boundary circle. The Hamiltonian is
$$\mathcal H(A)=\frac{1}{2}\int_{\text{outer }S^1}\la A_\sigma,\mathbf V A_\sigma\ra\,d\sigma$$
where $\mathbf V:\g\to\g$ is the reflection w.r.t.\ $V_+$, $\sigma$ is the coordinate (angle) on the outer $S^1$, and $A|_{\text{outer }S^1}=A_\sigma d\sigma$.

This description does depend on $\h$, but a suitable reduction (removing only finitely many degrees of freedom), when we constrain the holonomy to be 1, i.e. when we replace the annulus with a disk, uses only $\g$ and $V_+$.
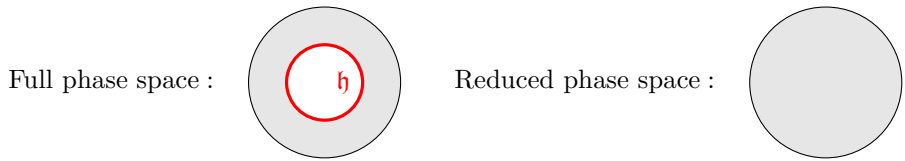
$$
\text{Full phase space}:\quad
\begin{tikzpicture}[scale=0.5, baseline=-0.5ex]
\filldraw[draw=black,fill=black!10!white](0,0)circle(2cm);
\filldraw[very thick,draw=red,fill=white](0,0)circle(1cm);
\node[red] at (0.5,0) {$\h$};
\end{tikzpicture}
\qquad
\text{Reduced phase space}:\quad 
\begin{tikzpicture}[scale=0.5, baseline=-0.5ex]
\filldraw[draw=black,fill=black!10!white](0,0)circle(2cm);
\end{tikzpicture}
$$

For details, see \cite{KS2} (the original explanation) and \cite{SCS} (the picture presented here and its space-time version).
\subsubsection{Ricci flow}
As observed in Example \ref{ex:ricGH}, we have (for constant sections of $\g\times \mathsf G/\mathsf H$)
$$\gric^{\g\times\mathsf G/\mathsf H}_{V_+\times\mathsf G/\mathsf H,\dv}=\gric^\g_{V_+,0}$$
In particular, the generalized Ricci flow of $V_+\times\mathsf G/\mathsf H$ in the exact CA $\g\times\mathsf G/\mathsf H$ (with the divergence $\dv$), i.e.\ the 1-loop renormalization group flow of $(g,H)$, is (pointwise $\mathsf G/\mathsf H$) the same as the generalized Ricci flow of $V_+$ in $\g$ (with the zero divergence).

\subsubsection{(Generalized) string background equations}\label{sec:gsbe}
As observed in Example \ref{ex:sugraGH}, the generalized metric $V_+\times\mathsf G/\mathsf H$ and the divergence $\dv$, or equivalently the quadruple $(g,H,X,\alpha)$, satisfies the generalized string background equations 
\begin{subequations}\label{PLsugI}
\begin{equation}
\gric^{\g\times\mathsf G/\mathsf H}_{V_+\times\mathsf G/\mathsf H,\dv}=0,\quad \mathcal R^{\g\times\mathsf G/\mathsf H}_{V_+\times\mathsf G/\mathsf H,\dv}=0
\end{equation}
 iff 
\begin{equation}
\gric^\g_{V_+,0}=0,\quad \mathcal R^\g_{V_+,0}=0.
\end{equation}
\end{subequations}
 Moreover, if $\h$ is unimodular so that there is a $\g$-invariant density $\mu$ on $\mathsf G/\mathsf H$, the generalized string background equations become the ordinary string background equations, with the dilaton $\phi=\frac{1}{2}\log\frac{\mu_g}{\mu}$. (If $\mathcal R^\g_{V_+,0}\neq0$, we can still interpret the outcome as the string background equations for a non-critical string).

As a minor generalization, if $\h$ is coisotropic instead of Lagrangian (i.e.\ if $\h^\perp\subset\h$ instead of $\h^\perp=\h$) then the CA $\g\times\mathsf G/\mathsf H$ is transitive instead of exact. If $\dim V_+=\dim\g/\h$ then the generalized metric 
$V_+\times \mathsf G/\mathsf H\subset \g\times \mathsf G/\mathsf H$ (together with $\dv$) is equivalent to the (bosonic) field content of the type I SUGRA together with the gauge fields, and  \eqref{PLsugI}
is again equivalent to the SUGRA equations \cite{MGF2}. In other words, PL T-duality works also in this case.

\subsection{CA pullbacks and PL T-duality with spectators}
Poisson-Lie T-duality is most easily formulated in terms of pullbacks of CAs in the following sense:
\begin{defn}
If $E\to M$ is CA and $\tau:M'\to M$ is a smooth map, we shall say that $E':=\tau^*E$ is a \emph{CA-pullback} of $E$  if on $E'$ we have a compatible CA structure: 
\begin{gather*}\label{pullback-CA}
\la \tau^*u,\tau^*v\ra'=\tau^*\la u,v\ra,\quad [\tau^*u,\tau^*v]'=\tau^*[u,v]\\
\tau_*\bigl(\rho'(\tau^*u)\bigr)=\rho(u)\quad\forall u,v\in\Gamma(E).
\end{gather*}
\end{defn}

\begin{rem}
While $\la,\ra'$ and $[,]'$ are determined by $\tau$, $\rho'$ is not.
CA-pullbacks were characterized by Li-Bland and Meinrenken \cite{LBM} as follows:  a compatible CA structure on $\tau^*E$ is uniquely specified by its anchor map $\rho':\tau^*E\to TM'$, and it exists iff $\rho'$ satisfies
\begin{itemize}
\item $\tau_*\bigl(\rho'(\tau^*u)\bigr)=\rho(u)\quad\forall u\in\Gamma(E)$
\item $[\rho'(\tau^*u),\rho'(\tau^*v)]=\rho'(\tau^*[u,v])\quad\forall u,v\in\Gamma(E)$ 
\item for any $p\in M'$ the kernel of $\rho'$ at $p$ is a $\la,\ra$-coisotropic subspace of $E_{f(p)}$.
\end{itemize}
\end{rem}

\begin{example}
Let us consider the case $M=\text{point}$:
 $E=\g$ is a Lie algebra with an invariant pairing $\la,\ra$. 
Let $M'$ be a manifold with an action $\rho'$ of $\g$ such that the stabilizers of points are coisotropic. Then \cite{LBM} $E':=\g\times M'$ is a CA: the pairing and the bracket of constant sections is the pairing and the bracket on $\g$, and the anchor map is the action $\rho'$.

In particular, if $M'=\mathsf G/\mathsf H$ with $\h^\perp=\h$, then $E'$ is an exact CA. These exact CA-pullbacks of $\g$ correspond to PL T-duality without spectators, as we discussed in detail above.
\end{example}

The main source of examples giving exact $E'$'s is described below in Example \ref{ex:spect}.

If now $\tau^*E$ is a CA-pullback and $V_+\subset E$ a generalized metric then $\tau^*V_+\subset \tau^*E$ is a generalized metric. Moreover, if $\dv$ is a divergence on $E$ then there is a unique divergence $\tau^*\dv$ on $\tau^*E$ characterized by
$$(\tau^*\dv)(\tau^*u)=\tau^*(\dv u)\qquad\forall u\in\Gamma(E).$$
If $\dv$ is compatible with $V_+$ then $\tau^*\dv$ is compatible with $\tau^*V_+$.

Poisson-Lie T-duality is then the statement that various properties of $\tau^*V_+\subset \tau^*E$ and of $\tau^*\dv$ are determined by $V_+\subset E$ and by $\dv$ (the spectator-less case corresponds to $E=\g$ and $\dv=0$).

\subsubsection{Sigma models}\label{sec:sigmaPL}

Let us suppose that $\tau:M'\to M$ is a surjective submersion and that $E'=\tau^*E\to M$ is an exact CA (which forces $E\to M$ to be transitive), so that the generalized metric $\tau^*V_+\subset \tau^*E$ gives rise to a pair $(g,H)$ on $M'$, which can be used to define a 2-dim $\sigma$-model with the target space $M'$. 

The phase space of the $\sigma$ model is $T^*(LM')$ with the symplectic form twisted by $H$. As in the no-spectator case, Poisson-Lie T-duality says that there is an infinite-dimensional symplectic manifold defined in terms of the CA $E\to M$, and a Hamiltonian defined in terms of $V_+$, which is at the same time a finite-codimension coisotropic reduction of $T^*(LM')$. Up to finitely many degrees of freedom the $\sigma$-model can thus be formulated in terms of $V_+\subset E$, without any reference to $M'$ and $E'$.

We shall not give a proof of this statement in this paper, as we're more concerned here with the Ricci flow and the string background equations. Instead, we refer the reader to \cite{KS,KS2,SCS}. 

\begin{rem}
The main idea of \cite{SCS} is to use $V_+\subset E$ to define a 2-dim $\sigma$-model as a 3-dim topological field theory (Courant $\sigma$-model) with a non-topological boundary condition given by $V_+$. For exact CAs it gives the standard 2-dim $\sigma$-models \eqref{sigmaaction}. Conjecturally the generalized Ricci flow is the renormalization group flow of this  exotic $\sigma$-model and all other ``generalized things'' that we consider, such as the generalized string background equations or the generalized SUGRA equations (see \S\ref{sec:SUGRA}), can be interpreted in these terms.
\end{rem}

\subsubsection{Ricci flow}
Poisson-Lie T-duality for (generalized) Ricci flow is the following simple statement.
\begin{thm}\label{thm:PLric}
If $E\to M$ is a CA, $V_+\subset E$ a generalized metric, $\dv$ a divergence on $E$,  and $\tau^*E$ a CA-pullback then $\gric^{\tau^*E}_{\tau^*V_+,\tau^*\dv}=\tau^*\gric^E_{V_+,\dv}$, i.e.
$$\gric^{\tau^*E}_{\tau^*V_+,\tau^*\dv}(\tau^*a_+,\tau^*b_-)=\tau^*(\gric^E_{V_+,\dv}(a_+,b_-))\quad\forall a_+\in\Gamma(V_+),b_-\in\Gamma(V_-)$$
\end{thm}
\begin{proof}
Immediate from the definition of $\gric$.
\end{proof}

\subsubsection{(Generalized) string background equations}

Poisson-Lie T-duality for (generalized) string background equations follows from Theorem \ref{thm:PLric} and from the following equally easy result.
\begin{thm}\label{thm:PLR}
If $E\to M$ is a CA, $V_+\subset E$ a generalized metric, $\dv$ a divergence on $E$,  and $\tau^*E$ a CA-pullback then
$$\mathcal R^{\tau^*E}_{\tau^*V_+,\tau^*\dv}=\tau^*(\mathcal R^E_{V_+,\dv})$$
\end{thm}
\begin{proof}
Immediate from Equation \eqref{genR}, if for the local frame of $\tau^*E$ we use $\tau^*e_\alpha$.
\end{proof}

As a result, if $V_+\subset E$ and $\dv$ satisfy the generalized string background equations then so do $\tau^*V_+\subset \tau^*E$ and $\tau^*\dv$.

\begin{rem}\label{rem:mugen}
The problem of the \emph{existence of a dilaton}  boils down to the following: if $\dv=\dv_{\sigma^2}$ for a half-density $\sigma$ on $M$, is $\tau^*\dv=\dv_{(\sigma')^2}$ for some half-density $\sigma'$ on $M'$? The answer is as follows. Let us suppose (as is true in physically relevant situations) that $\tau:M'\to M$ is a submersion. Then $\sigma'=\varsigma\,\tau^*\sigma$ for some fibrewise half-density $\varsigma$ on the fibres of $\tau:M'\to M$, and $\tau^*\dv=\dv_{(\sigma')^2}$ iff 
$$\mathcal L_{\rho'(\tau^*u)}\varsigma=0\qquad\forall u\in\Gamma(E)$$
(this equality makes sense as $\rho'(\tau^*u)$ are fibration-preserving vector fields on $M'$). 

Alternatively, we have 
$$\Delta^{E'}_{V_+'}(\varsigma\,\tau^*\sigma)=\varsigma\,\tau^*\Delta^E_{V_+}\sigma$$
for every half-density $\sigma$ on $M$, so if $\Delta^E_{V_+}\sigma=0$ and $\sigma'=\varsigma\,\tau^*\sigma$ then $\Delta^{E'}_{V_+'}\sigma'=0$.
\end{rem}

\begin{rem}
We formulated PL T-duality in terms of CA-pullbacks. A more general, and arguably more conceptual approach is via Dirac relations (see e.g.\ \cite{LBM} for definitions). If $V^{(1)}_+\subset E^{(1)}$ and $V^{(2)}_+\subset E^{(2)}$ are generalized metrics and if $C\to N$ is a Dirac structure in $\overline{E^{(1)}}\times E^{(2)}$ (where $N\subset M^{(1)}\times M^{(2)}$ is a submanifold) then $C$ gives rise to a Lagrangian relation between the phase spaces given by $E^{(1)}$ and $E^{(2)}$. Moreover, if $C$ is compatible with $V^{(1)}_+$ and $V^{(2)}_+$ then this Lagrangian relation is, up to finitely many degrees of freedom, an isomorphism of Hamiltonian systems.

We leave the problem of compatibility of this more general T-duality with the Ricci flow and with the string background equations to a future work; one important problem is to find a source of examples which are not examples of (possibly equivariant) PL T-duality. The existence of the dilaton should correspond to unimodularity of the Dirac structure $C$.
\end{rem}

\subsection{Equivariant CAs and reduction}

If $\mathfrak k$ is a Lie algebra with a (possibly degenerate) invariant symmetric bilinear pairing $\la,\ra_{\mathfrak k}$, a \emph{$(\mathfrak k,\la,\ra_{\mathfrak k})$-equivariant CA} is a CA $E\to M$ together with a linear map $\chi:\mathfrak k\to\Gamma(E)$ satisfying
$$[\chi(u),\chi(v)]=\chi([u,v]),\quad\la\chi(u),\chi(v)\ra=\la u,v\ra_{\mathfrak k}$$
which is injective at every point of $M$ (the last condition is void if $\la,\ra_{\mathfrak k}$ is non-degenerate).
The derivations $[\chi(u),\cdot]$ give an action of ${\mathfrak k}$ on $E$ and the vector fields $\rho(\chi(u))$ an action of ${\mathfrak k}$ on $M$. If this action integrates to an action of a connected Lie group $\mathsf K$ with the Lie algebra ${\mathfrak k}$, we shall say that $E$ is $(\mathsf K,\la,\ra_{\mathfrak k})$-equivariant.

Equivariant CAs can be reduced in the following way \cite{red,PLC}. If  $E\to M$ is a $\mathsf K$-equivariant CA such that the  action of $\mathsf K$ on $M$ is free and proper, let  
$$(E_{/\mathsf K})_x:=(\chi_x({\mathfrak k}))^\perp/\chi_x({\mathfrak k}')\quad(\forall x\in M)$$
where ${\mathfrak k}'\subset{\mathfrak k}$ is the kernel of $\la,\ra_{\mathfrak k}$. After taking the quotient by $\mathsf K$, $E_{/\mathsf K}$ becomes a vector bundle $E_{/\mathsf K}\to M/\mathsf K$, and the CA structure on $E\to M$ descends to a CA structure on $E_{/\mathsf K}\to M/\mathsf K$. If $E$ is exact and $\la,\ra_{\mathfrak k}=0$ then $E_{/\mathsf K}$ is also exact; for a general $\la,\ra_{\mathfrak k}$ the CA $E_{/\mathsf K}$ is only transitive (i.e.\ its anchor map is surjective).

\begin{example}\label{ex:spect} 
Let us suppose that a connected Lie group $\mathsf G$ acts freely and properly on a manifold $P$, so that $P\to M:=P/\mathsf G$ is a principal $\mathsf G$-bundle, and that $\la,\ra_\g$ is a non-degenerate invariant pairing (i.e.\ that $\g$ is a CA over a point). Then $\mathsf G$-equivariant exact CAs $E_P\to P$ exist iff the 1st Pontryagin class $[\la F,F\ra_\g]\in H^4(M,\R)$ of $P\to M$ vanishes, and are classified by classes $\omega\in\Omega^3(M)/d\Omega^2(M)$ such that $d\omega=\la F,F\ra_\g$.

In this case $E:=(E_P)_{/\mathsf G}\to M$ is a transitive CA. If $\mathsf H\subset\mathsf G$ is such that $\h^\perp=\h$ then $E/_{\mathsf H}\to P/\mathsf H$ is an exact CA and moreover $E/_\mathsf H$ is naturally a CA-pullback of $E\to M$ under the projection $P/\mathsf H\to M$. This is the main source of CA-pullbacks for the purpose of PL T-duality. See \cite{PLC} for details. (If $\h$ is coisotropic, i.e.\ $\h^\perp\subset\h$, then $E/_{\mathsf H}$ is transitive, and it is still a CA-pullback of $E$.)
\end{example}

\emph{From now on we shall always suppose that $\s$ is a Lie algebra with $\la,\ra|_\s=0$ (which implies $\s'=\s$) and that $\mathsf S$ is a connected Lie group integrating $\s$.}

If $E$ is $\mathsf S$-equivariant and if the action of $\mathsf S$ is free and proper then
\begin{equation}\label{identif}C^\infty(M/\mathsf S)\cong C^\infty(M)^\s,\qquad\Gamma(E_{/\mathsf S})\cong \Gamma(\s_E^\perp)^\s/\Gamma(\s_E)^\s,\end{equation}
where $\s_E$ and $\s_E^\perp$ are the subbundles of $E$ given fiberwise by $\chi(\s)$ and $\chi(\s)^\perp$ respectively, and superscript $\s$ means we are considering invariant sections, e.g.
\[C^\infty(M)^\s=\{f\in C^\infty(M)\mid \chi(s)\cdot f=0\quad\forall s\in \s\},\]
\[\Gamma(\s_E^\perp)^\s=\{x\in \Gamma(\s^\perp_E)\mid [\chi(s),x]=0 \quad \forall s\in \s\}.\]

If $E\to M$ is an $\s$-equivariant CA and $\dv$ a divergence on $E$, we shall say that $\dv$ is \emph{equivariant} if 
\[\chi(s)\cdot \dv x=\dv[\chi(s),x],\quad\forall s\in\s,\,\forall x\in\Gamma(E).\] 
and if
\begin{equation}\label{unimod}
\dv \chi(s)=-\on{Tr}_\s ad_s\quad \forall s\in \s.
\end{equation}

\begin{prop}\label{divdescends}
Suppose $E$ is $\mathsf S$-equivariant such that the action of $\mathsf S$ is free and proper. If $\dv$ is an equivariant divergence  on $E$ then it descends, via the identification \eqref{identif}, to a divergence on the reduced CA $E/_{\mathsf S}$.
\end{prop}
\begin{proof}
Equivariance implies that $\dv$ restricts to a map $\Gamma(E)^\s\to C^\infty(M)^\s$. We thus only have to show that $\dv \Gamma(\s_E)^\s=0$.

Suppose $u_i$ is a basis of $\s$ and consider $x=x^i \chi(u_i)\in\Gamma(\s_E)^\s$, $x^i\in C^\infty(M)$. By $\s$-invariance, for all $i$ we have
$0=[\chi(u_i),x]=x^j[\chi(u_i),\chi(u_j)]+(\chi(u_i)\cdot x^j)\chi(u_j)$,
so $\chi(u_i)\cdot x^i=x^j\on{Tr}_\s ad_{u_j}$. Now
\[\dv x=x^i\dv \chi(u_i)+\chi(u_i)\cdot x^i=x^i(\dv \chi(u_i)+\on{Tr}_\s ad_{u_i})=0.\]
\end{proof}
\begin{rem}
If $\dv$ comes from a density on $M$ then the induced divergence on $E_{/\mathsf S}$ also comes from a density on $M/\mathsf S$.
\end{rem}

Finally, let us show that $\gric$  and $\mathcal R$ are compatible with the CA reductions.
\begin{defn}
Let $E\to M$ be an $\s$-equivariant CA.  A generalized pseudometric $V_+\subset E$ is \emph{admissible} if it is $\s$-invariant and if
\begin{equation}\label{V+red}
V_+\subset\s_E^\perp
\end{equation}
(i.e.\ if $\la\chi(s),a_+\ra=0$ for all $s\in\s,a_+\in\Gamma(V_+)$).
\end{defn}

If the action of $\mathsf S$ is free and proper, an admissible $V_+\subset E$ descends to a generalized pseudometric  $\tilde V_+\subset E_{/\mathsf S}$ (of the same rank as $V_+$); we have 
$$\Gamma(\tilde V_+)\cong\Gamma(V_+)^\s,\quad \Gamma(\tilde V_-)\cong\Gamma(V_-\cap\s^\perp_E)^\s/\Gamma(\s_E)^\s$$ 
Let us choose an equivariant divergence $\dv$ on $E$ and let $\widetilde\dv$ be the corresponding divergence on $E_{/\mathsf S}$.

\begin{thm}\label{thm:gricred}
Under the above assumptions, if $a_+\in\Gamma(V_+)^\s$ and $b_-\in\Gamma(V_-\cap\s^\perp_E)^\s$, and if $\tilde a_+$ and $\tilde b_-$ are the corresponding sections of $\tilde V_+$ and $\tilde V_-$, then
$$\gric^E_{V_+,\dv}(a_+,b_-)=\pi^*\gric^{E_{/\mathsf S}}_{\tilde V_+,\widetilde\dv}(\tilde a_+,\tilde b_-)$$
where $\pi:M\to M/\mathsf S$ is the projection.
In other words, reduction of admissible generalized pseudometrics is compatible with the generalized Ricci flow.
\end{thm}
\begin{proof}
This is an immediate corollary of the fact that all the structures (the CA bracket, generalized metric, divergence, etc.) involved in the definition of $\gric$ on $E_{/\mathsf S}$ are induced by the corresponding structures on $E$.
\end{proof}

\begin{thm}\label{thm:Rred}
Under the above assumptions
$$\mathcal R^E_{V_+,\dv}=\pi^*\mathcal R^{E_{/\mathsf S}}_{\tilde V_+,\widetilde\dv}$$
where $\pi:M\to M/\mathsf S$ is the projection.
\end{thm}
\begin{proof}
We can choose the local frame $e_a$ of $V_+$ to be $\s$-invariant, so that it descends to a local frame of $\tilde V_+$. The first 3 terms in the formula \eqref{genR} for $\mathcal R$ are then the same in $E$ and in $E_{/\mathsf S}$, so we must only concentrate on the last term 
\begin{equation}\label{ccred}
c_{ab\bar c}c^{ab\bar c}=\la[e_a,e_b]_-,[e^a,e^b]_-\ra.
\end{equation}

For any $s\in\s$ we have 
$$\la\chi(s),[e_a,e_b]\ra=e_a\cdot\la\chi(s),e_b\ra-\la [e_a,\chi(s)],e_b \ra=0$$
and thus $[e_a,e_b]_-\in\Gamma(\s_E^\perp)$. This implies that \eqref{ccred} gives the same result in $E$ and in $\pi^*(E_{/\mathsf S})=\s_E^\perp/\s_E$.
\end{proof}

\begin{rem}
A similar result is valid for the Laplace operator $\Delta_{V_+}$: if we identify half-densities on $M/\mathsf S$ with equivariant half-densities on $M$ (this requires a choice of a non-zero element of $|\bigwedge^{top}\s|^{1/2}$), then $\Delta_{V_+}$ restricted to the equivariant half-densities is equal to $\Delta_{\tilde V_+}$. We leave the details to the reader.
\end{rem}

\subsection{Dressing cosets, or equivariant Poisson-Lie T-duality}\label{sec:equiv}

\subsubsection{Outlook}

Poisson-Lie T-duality, as we described it so far, was about CAs and CA-pullbacks (and pullbacks of generalized metrics and divergences). Its original motivation is the study of 2-dim $\sigma$-models.

The same strategy can be followed for equivariant CAs (and admissible generalized pseudometrics and equivariant divergences); on the Physics side $\sigma$-models get replaced with gauged $\sigma$-models.

We shall then reduce the equivariant CAs obtained by CA-pullbacks, and get just plain CAs (and generalized (pseudo)metrics and divergences); this can be done under the assumption of free and proper actions. The gauged $\sigma$-models then become ordinary $\sigma$-models on the quotient targets.

This type of equivariant PL T-duality was introduced in \cite{KS3} under the name ``PL T-duality of dressing cosets''.

\subsubsection{Setup}

Let $E\to M$ be an $\s$-equivariant CA, $V_+\subset E$ an admissible generalized pseudometric, and $\dv$ an equivariant divergence.

\begin{example}\label{ex:eqPL}
The simplest type of our setup is $M=\text{point}$, $E=\g$ is a quadratic Lie algebra, $\dv=0$,  $\s\subset\g$ a unimodular isotropic Lie subalgebra, and $V_+\subset \g$ a $\s$-invariant vector subspace s.t.\ $V_+\subset \s^\perp$ and s.t.\ $\la,\ra|_{V_+}$ is non-degenerate.
\end{example}

Suppose now that $\tau: M'\to M$ is a smooth map and that we have a CA-pullback structure on $E':=\tau^*E\to M'$. Then $E'$ is also $\s$-equivariant, with $\chi'(s)=\tau^*\chi(s)$ ($\forall s\in\s$). Moreover $\tau^*V_+$ is admissible and $\tau^*\dv$ is equivariant.

If the action of $\s$ on $M'$ integrates to a free and proper action of $\mathsf S$ on $M'$  then we get the reduced CA $E'_\redu:=(\tau^*E)/_{\mathsf S}\to M'/\mathsf S=:M'_\redu$, and
$\tau^* V_+$ and $\tau^*\dv$ descend to a generalized pseudometric $V'_{+,\redu}\subset E'_\redu$ (of the same rank as $V_+$) and to a divergence $\dv'_\redu$ in $E'_\redu$.

\emph{Equivariant PL T-duality} is then a collection of statements that various properties of $(V'_{+,\redu}\subset E'_\redu,\dv'_\redu)$ can be expressed in terms of $(V_+\subset E,\dv,\chi)$.

\subsubsection{Sigma models}
As $\sigma$-models as Hamiltonian systems are not the subject of  this paper, we just describe the main idea: If $E$ is $\mathsf S$-equivariant, we get a Hamiltonian action of the loop group $L\mathsf S$ on the corresponding phase space. The phase space corresponding to $E'_\redu$ is then obtained from the one corresponding to $E'$ by the symplectic reduction w.r.t.\ $L\mathsf S$. One thus just needs to combine the results of  \S\ref{sec:sigmaPL} with the symplectic reduction by $L\mathsf S$.

\subsubsection{Ricci flow} Let $\mathcal M_E$ be the space of all generalized pseudometrics on $E$ and $\mathcal N_E\subset\mathcal M_E$ the subset of admissible generalized pseudometrics.

Recall that $\gric^E_{\dv}$ can be seen as a vector field on $\mathcal M_E$.

\begin{prop}
$\gric^E_{\dv}$, seen as a vector field on $\mathcal M_E$, is tangent to $\mathcal{N}_E\subset\mathcal M_E$.
\end{prop}
\begin{proof}
The vector field is certainly tangent to the subspace of all $\s$-invariant generalized metrics. To show that it's tangent to $\mathcal N_E$ we thus need to show that
\begin{equation}\label{gricvan}
\gric_{V_+,\dv}(a_+,\chi(s))=0
\end{equation}
for any admissible $V_+$.
As $\dv$ is $\s$-equivariant and $V_+$ is $\s$-invariant, we have 
$$\chi(s)\cdot\dv a_+=\dv[\chi(s),a_+]=\dv[\chi(s),a_+]_+$$
so the first two terms in \eqref{gricci} (with $b_-=\chi(s)$) cancel. The last term vanishes, as $[\cdot,\chi(s)]_-=0$ by $\s$-invariance of $V_+$.
\end{proof}

Now we can formulate the compatibility of the (generalized) Ricci flow with equivariant PL T-duality. As we observed above, every generalized pseudometric $V_+\in\mathcal N_E$ gives rise to a generalized pseudometric $V'_{+,\redu}\in\mathcal M_{E'_\redu}$; let $\varphi:\mathcal N_E\to\mathcal M_{E'_\redu}$ denote the map $V_+\mapsto  V'_{+,\redu}$.

\begin{thm}\label{thm:redric}
The vector field $\mathrm{GRic}^E_{\dv}$ on $\mathcal{N}_E$ pushes forward, via $\varphi$, to the vector field $\mathrm{GRic}^{E'_\redu}_{\dv'_{\mathrm{red}}}$ on $\mathcal M_{E'_\redu}$. 
\end{thm}
\begin{proof}
This is an immediate corollary of Theorems \ref{thm:PLric} and \ref{thm:gricred}.
\end{proof}

\subsubsection{(Generalized) string background equations}
Finally, we have the following result stating that equivariant PL T-duality is compatible with the generalized string background equations.
\begin{thm} If an admissible $V_+\subset E$ and an equivariant $\dv$ satisfy the generalized string background equations then so do  $V'_{+,\redu}\subset E'_\redu$ and $\dv'_\redu$.
\end{thm}
\begin{proof}
This follows immediately from the previous Theorem (for $\gric=0$) and from Theorems \ref{thm:PLR} and \ref{thm:Rred} (for $\mathcal R=0$).
\end{proof}

\begin{rem}
In \cite{KS3}, in place of the subbundle $V_+$ used here, a different subbundle $W_+\subset E$ was used; it relates to our $V_+$ via $W_+=V_+ +\s_E$ and is characterized by the property $\s_E=\ker\la,\ra|_{W_+}\subset W_+$ and by being $\s$-invariant. While from the point of view of the $\sigma$-model $W_+$ is more natural than $V_+$, we were not able to prove compatibility of the equivariant PL T-duality with the Ricci flow and with string background equations using only $W_+$. Whether it's possible remains an open problem and needs new ideas ($W_+\subset E$ is not a generalized pseudometric).
\end{rem}

\section{Differential graded symplectic manifolds, spinor bundles and Dirac generating operators}\label{dgsection}

In this section we recall Dirac generating operators of CAs and motivate them using the graded symplectic interpretations of CAs. Generating Dirac operators were introduced in \cite{AX}; here we follow the approach of \cite[letter 6]{let}, where the canonical generating operator was first constructed.

(For a reader feeling that it's taking us off the main road it's enough to look at the formulas \eqref{gendir} and \eqref{gendirdiv} for the Dirac operators and browse through \S\ref{spinorbundle} and \S\ref{sec:spinred} to understand the notation.)

To avoid possible confusion, the Courant bracket in a CA $E$ will be denoted in this section by $[,]_E$, as $[,]$ will be the (graded) commutator in a (graded) associative algebra.

\subsection{CAs as dg symplectic manifolds}\label{sec:dg}
An alternative viewpoint at CAs is provided by graded geometry \cite{DR}. Namely, a vector bundle $E$ with a non-degenerate symmetric pairing $\la,\ra$ is equivalent to a non-negatively graded manifold $\mathcal E$ with a symplectic form $\omega$ of degree 2, and a CA structure on $E$ is equivalent to a degree 3 function $\Theta$ on $\mathcal E$ satisfying the master equation
$$\{\Theta,\Theta\}=0.$$

The correspondence is as follows.
If $\mathcal{E}$ is such a graded manifold then the vector bundle $E\to M$ is given by
\[\Gamma(E)=C^\infty(\mathcal{E})^1,\quad C^\infty(M)=C^\infty(\mathcal{E})^0\]
with the pairing 
$$\langle u,v\rangle =\{u,v\},$$
and the CA structure is
\begin{align*}
[u,v]_E&=\{\{\Theta, u\},v\}\\
\rho(u)f&=\{\{\Theta,u\},f\},
\end{align*}
where $u,v\in\Gamma(V)$ and $f\in C^\infty(M)$. We also have $d_Ef=\{\Theta,f\}$.

If we (locally) trivialize the bundle $E\to M$ to $V\times M$ for some vector space $V$ with a pairing $\la,\ra$, we get $\mathcal E=V[1]\times T^*[2]M$, with the standard symplectic structure on $T^*[2]M$ and with the constant symplectic structure on $V[1]$ given by $\la,\ra$.
Using this trivialization and a basis $e_\alpha$ of $V\cong V^*$ we then have
\[\Theta=e^\alpha\rho(e_\alpha)-\frac{1}{6}e^\alpha e^\beta e^\gamma c_{\alpha\beta\gamma}
\in{\textstyle\bigwedge}V\otimes C^\infty(T^*[2]M)=C^\infty(\mathcal E)\]
where  $c_{\alpha\beta\gamma}=\langle [e_\alpha,e_\beta]_E,e_\gamma\rangle$.

\subsection{Differential operators with Clifford coefficients}

If $V$ is a finite-dimensional real vector space with a non-degenerate symmetric pairing $\la,\ra$, let $\on{Cl}(V)$ be the corresponding Clifford algebra, i.e.\ the algebra generated by $V$ modulo the relations
$$uv+vu=\la u,v\ra.$$

Let us recall some basic properties of the algebra $\on{Cl}(V)$. It is $\Z_2$-graded and equipped with an increasing filtration; its associated graded is $\bigwedge V$. The complex version $\on{CL}(V)\otimes\C$ is equipped with a ``$\Z_2$-graded involution'' 
$$\alpha\mapsto\alpha^T,\quad (\alpha^T)^T=(-1)^{|\alpha|}\alpha,\quad(\alpha\beta)^T=(-1)^{|\alpha||\beta|}\beta^T\alpha^T,$$
determined by $v^T=iv$ ($\forall v\in V$).

If $M$ is a manifold and $t\in\R$, let $D_t(M)$ be the algebra of differential operators  acting on $t$-densities on $M$. This algebra is naturally filtered by the orders of differential operators. On $D_{1/2}(M)$ we have a natural involution $D\mapsto D^T$ where $D^T$ is the formally adjoint operator (i.e.\ $\int_M\sigma_1 D\sigma_2=\int_M \sigma_2 D^T\sigma_1$ for any compactly supported half-densities $\sigma_1$, $\sigma_2$).

Let us now combine Clifford algebras and differential operators to get a (filtered) deformation quantization of the graded Poisson algebra $C^\infty(\mathcal E)$.
 If $E\to M$ is a vector bundle with a non-degenerate symmetric pairing $\la,\ra$ and  $t\in\R$, we shall define an algebra $\mathcal A_t(E)$  containing $\Gamma(\on{Cl}(E))$.

If we have a trivialization $E\cong V\times M$, where $V$ is a vector space with a pairing $\la,\ra$, we set 
$$\mathcal A_t(E)=\on{Cl}(V)\otimes D_t(M).$$
If we change the trivialization by a gauge transformation $g:M\to O(V)$, we (locally) lift $g$ to $\tilde g:M\to\on{Pin}(V)\subset\on{Cl}(V)$, and act on $\mathcal A_t(E)$ with $Ad_{\tilde g}$. This allows us to define $\mathcal A_t(E)$ for non-trivial bundles $E$, by gluing it from local $\on{Cl}(V)\otimes D_t(U_i)$'s using the transitions $g_{ij}:U_i\cap U_j\to O(V)$.

\begin{rem}
An invariant definition of $\mathcal A_t(E)$ is as follows. Let $P\to M$ be the principal $O(V)$-bundle associated to $E$ (the points of $P$ are $\la,\ra$-preserving linear bijections $V\to E_x$, $x\in M$). Using the standard inclusion of $\mathfrak{so}(V)$ to $\on{Cl}(V)$ and to the vector fields on $P$ we get the inclusion
$$\mathfrak{so}(V)\to \on{Cl}(V)\otimes D_t(P),\quad X\mapsto  X\otimes 1+1\otimes \mathcal L_X.$$
We then have
$$\mathcal A_t(E)=\bigl(\on{Cl}(V)\otimes D_t(P)/\la \mathfrak{so}(V)\ra\bigr)^{O(V)}$$
where $\la \mathfrak{so}(V)\ra\subset \on{Cl}(V)\otimes D_t(P)$ is the ideal generated by $\mathfrak{so}(V)$.

This construction is a quantum version of a construction of $\mathcal E$ as a symplectic reduction, namely $\mathcal E=\bigl(T^*[2]P\times V[1]\bigr)/\!/O(V)$ (due to A.~Weinstein).
\end{rem}

The algebra $\mathcal A_t(E)$ is $\Z_2$-graded, as $\on{Cl}(V)$ is $\Z_2$-graded.
It comes with a natural increasing filtration, given (using a local trivialization $E\cong V\times M$) by the filtration degree in $\on{Cl}(V)$ plus \emph{twice} the filtration degree (i.e.\ order) in $D_t(M)$ (this combination is invariant under the transitions $Ad_{\tilde g}$), i.e.
$$F^p\mathcal A_t(E)=\sum_{r+2s=p}F^r\on{Cl}(V)\otimes F^sD_t(M).$$
In particular 
$$F^0\mathcal A_t(E)=C^\infty(M),\quad F^1\mathcal A_t(E)^{odd}=\Gamma(E).$$
 The associated graded Poisson algebra is $C^\infty(\mathcal E)$ where $\mathcal E$ is the graded symplectic manifold corresponding to $(E,\la,\ra)$ (see \S\ref{sec:dg})

Finally, the involutions on $\on{Cl}(V)\otimes\C$ and on $D_{1/2}(M)$ give us a $\Z_2$-graded involution on $\mathcal A_{1/2}(E)\otimes\C$. This $\Z_2$-graded involution induces on
 $$C^\infty(\mathcal E)\otimes\C=\on{Gr}\mathcal A_{1/2}(E)\otimes\C$$
  the involution $\alpha^T= i^{|\alpha|}\alpha$.

For $X\in F^p\mathcal A_{1/2}(E)$ we shall call its image in 
$$C^\infty(\mathcal E)^p=F^p\mathcal A_{1/2}(E)/F^{p-1}\mathcal A_{1/2}(E)$$
 the \emph{principal symbol} of $X$. As the principal symbol $\on{symb}_p X$ of $X$ is of parity $p$ mod 2 and it satisfies $(\on{symb}_p X)^T=i^p\on{symb}_p X$, we have $\on{symb}_p X=0$ if $X$ is of the opposite parity or if $X^T=-i^p X$.

\subsection{Generating Dirac operator}

Let $E\to M$ be a CA, and $\Theta\in C^\infty(\mathcal E)$ the corresponding degree 3 function. Let us see whether there is a canonical lift of $\Theta$ to $\mathcal A_t(E)\otimes\C$.

For $t=1/2$ there is such a canonical lift:

\begin{prop}\label{prop:Dcan}
There is a unique odd element $\mathcal D\in F^3\mathcal A_{1/2}(E)$ satisfying 
$$\mathcal D^T=-i\mathcal D$$
 such that its principal symbol 
  is $\Theta$. Moreover $\mathcal D^2\in C^\infty(M)\subset\mathcal A_{1/2}(E)$.
\end{prop}
\begin{proof}
An element of $F^3\mathcal A_{1/2}(E)^{odd}$ is determined by its principal symbol up an  element $e\in F^1\mathcal A_{1/2}(E)^{odd}=\Gamma(E)$. As $e^T=i e$, demanding $\mathcal D^T=-i\mathcal D$ fixes $\mathcal D$ completely.

We have $\mathcal D^2=[\mathcal D,\mathcal D]/2\in F^2\mathcal A_{1/2}(E)^{even}$, and $(\mathcal D^2)^T=-(\mathcal D^T)^2=\mathcal D^2$, so $\mathcal D^2\in F^0\mathcal A_{1/2}(E)=C^\infty(M)$.
\end{proof}

The element $\mathcal D$ is called \emph{the generating Dirac operator}.
Using a (local) trivialization $E=V\times M$ we have
\begin{equation}\label{gendir}
\mathcal D = e^\alpha\mathcal L_{\rho(e_\alpha)} - \frac{1}{6}e^\alpha e^\beta e^\gamma c_{\alpha\beta\gamma}\in\on{Cl}(V)\otimes D_{1/2}(M).
\end{equation}

\begin{rem}\label{rem:D^2}
A quick calculation (where we can discard all terms satisfying $X^T=-X$ appearing in $\mathcal D^2$ as they have to cancel) gives
$$\mathcal D^2=\frac{1}{2}\mathcal L_{\rho(e_\alpha)}\mathcal L_{\rho(e^\alpha)}-\frac{1}{48}c_{\alpha\beta\gamma}c^{\alpha\beta\gamma}$$
i.e.\ $\mathcal D^2=\frac{1}{8}\Delta_E$. The function $\mathcal D^2$ vanishes for exact CAs. For a quadratic Lie algebra it is $-\frac{1}{48}c_{\alpha\beta\gamma}c^{\alpha\beta\gamma}$. Any transitive CA is locally $TM\oplus T^*M\oplus\g$ for a quadratic Lie algebra $\g$, and the function is again the constant $-\frac{1}{48}c_{\alpha\beta\gamma}c^{\alpha\beta\gamma}$.
\end{rem}

There is no canonical lift of $\Theta$ to $\mathcal A_0(E)$. However, given a non-vanishing half-density $\sigma$, we can identify $C^\infty(M)$ with half-densities via multiplication by $\sigma$, and thus also $D_0(M)$ with $D_{1/2}(M)$ and $\mathcal A_0(E)$ with $\mathcal A_{1/2}(E)$. Under this identification $\mathcal D\in \mathcal A_{1/2}(E)$ is sent to
$$D_\sigma:= e^\alpha \Bigl(\rho(e_\alpha)+\frac{1}{2}\dv_{\sigma^2} e_\alpha\Bigr) - \frac{1}{6}e^\alpha e^\beta e^\gamma c_{\alpha\beta\gamma}.$$
If $\dv$ is an arbitrary divergence on $E$ and $\sigma$ a non-vanishing half-density, and if $\dv=\dv_{\sigma^2}+\la e,\cdot\ra$, we set \cite{AX,MGF} 
$$D_{\dv}:=D_\sigma+e/2.$$
Using a local trivialization we then get
\begin{equation}\label{gendirdiv}
D_{\dv}= e^\alpha \Bigl(\rho(e_\alpha)+\frac{1}{2}\dv e_\alpha\Bigr) - \frac{1}{6}e^\alpha e^\beta e^\gamma c_{\alpha\beta\gamma}\in\on{Cl}(V)\otimes D_{0}(M)
\end{equation}
which also proves the independence of $D_{\dv}$ on the choice of $\sigma$.

\begin{rem}
We have $D_{\dv}^2=(D_\sigma+e/2)^2=D_\sigma^2+[D_\sigma,e]/2+\la e,e\ra/8$. This expression is in $C^\infty(M)$ iff $[D_\sigma,e]\in C^\infty(M)$, and that happens iff the derivation $[e,\cdot]_E$ of $E$ vanishes.
\end{rem}

\begin{example}
For the exact CAs of the type $\g\times\mathsf G/\mathsf H$, with the usual divergence $\dv u=0$ ($\forall u\in\g$) we have $D_{\dv}=e^\alpha\rho(e_\alpha)-\frac{1}{6}e^\alpha e^\beta e^\gamma c_{\alpha\beta\gamma}\in\on{Cl}(\g)\otimes D_{0}(\mathsf G/\mathsf H)$, which then gives
$D_{\dv}^2=\frac{1}{2}\rho(e_\alpha)\rho(e^\alpha)-\frac{1}{48}c_{\alpha\beta\gamma}c^{\alpha\beta\gamma}$. Let us recall that the differential operator $\rho(e_\alpha)\rho(e^\alpha)$ is actually a vector field.

On the other hand, as $\g\times\mathsf G/\mathsf H$ is exact, we have $\mathcal D^2=0$. If $\h$ is unimodular (so that $D_{\dv}=D_\sigma$ for an invariant half-density $\sigma$) we thus have $D_{\dv}^2=0$. As a result, if $c_{\alpha\beta\gamma}c^{\alpha\beta\gamma}\neq0$ then $\g$ contains no unimodular Lagrangian Lie subalgebras.
\end{example}

\subsection{Spinors}\label{spinorbundle}
 Let us suppose that $E$ admits a spin structure and let $S_E$ be the spinor bundle of modules of $\cl(E)$ (i.e.\ we suppose that there is a lift of the associated principal $O(V)$-bundle $P\to M$ to a $Pin(V)$-bundle $\tilde P\to M$ and $S_E$ is its associated bundle given by the spinor space $S_V$). The algebra $\mathcal A_0(E)$ then acts on $\Gamma(S_E)$ by differential operators. Likewise, $\mathcal A_{1/2}(E)$ acts on $\Gamma(\mathcal S_E)$ where 
$$\mathcal S_E:=S_E\otimes|{\textstyle\bigwedge}^{top}T^*M|^{1/2}$$
 is the bundle of spinor half-densities.

 We now recall a construction of the $\cl(V)$-module $S_V$. For simplicity we suppose that $\dim V$ is even and also that $\la,\ra$ has symmetric signature (the latter assumption can be avoided by a complexification).

Choosing a lagrangian splitting of $V$, i.e. $V=L_1\oplus L_2$, we set 
\begin{equation}\label{cliff}
S_V:={\textstyle\bigwedge} L_1\otimes ({\textstyle\bigwedge}^{\mathrm{top}}L_1)^{-1/2}.
\end{equation}
The action of $\cl(V)$ on $S_V$ is generated by the action of $L_1$ by multiplication and the action of $L_2\cong L_1^*$ by contraction (we act on the first factor in $S_V$). 

\begin{example}\label{ex:spinex}
If $E=TM\oplus T^*M$ is an exact CA given by a closed 3-form $H$ then we can choose $L_1=T^*M$ and $L_2=TM$ and get $S_E=\bigwedge T^*M\otimes (\bigwedge^{\mathrm{top}}T^*M)^{-1/2}$ and so $\mathcal S_E=\bigwedge T^*M$. In this representation we get $\mathcal D=d+H\wedge$.
\end{example}

Furthemore, we define the operator $\vartheta\colon \bw L_1 \to \bw L_1\otimes\C$ via $a\mapsto i^{|a|} a$ for $a$ a homogeneous element. 

We can now construct a natural pairing $(\cdot,\cdot): S_V\otimes S_V\to\C$, which, using the identification (\ref{cliff}), is defined as 
\[(a\otimes \lambda,b\otimes\lambda):=\langle(\vartheta a\wedge b)^{\mathrm{top}},\lambda^2\rangle,\]
for $a,b\in\bw L_1$ and $\lambda\in (\bw^{\mathrm{top}}L_1)^{-1/2}$. Here we have only kept the top part of the wedge product and then contracted it with $\lambda^2\in(\bw^{\mathrm{top}} L_1)^*$. 
One can easily check that the pairing is compatible with the action of $\cl(V)$, i.e.
\begin{equation}\label{mukai}
(u A,B)=(-1)^{|u||A|}(A,u^T B)\quad \forall A,B\in S_V, u\in \cl(V)
\end{equation}
and that 
\[(A,B)=i^{(2|A|-1)\on{dim}L_1}(B,A).\]

Finally, let $E$ be a CA and let 
 $V_+\subset E$ be a generalized pseudometric with a chosen orientation. Let $e_a$, $a\in\{1,\dots n\}$,  be a local oriented orthonormal basis of $V_+$ ($\la e_a,e_b\ra=\pm\delta_{ab}$). We define 
 $$R_{V_+}=2^{n/2}e_1\dots e_n\in\pin(E)\subset\cl(E).$$
  This is independent of the choice of basis.  $R_{V_+}$ is a lift of the reflection w.r.t. $V_+$ multiplied by $(-1)^{n+1}$ from $O(E)$ to $\pin(E)$. It is easy to check that
\[R^2_{V_+}=(-1)^{\left[\textstyle{\frac{n+1}{2}}\right]+q},\]
where $[\cdot]$ denotes the integer part and $q$ is the number of negative eigenvalues of the inner product on $V_+$. If $R^2_{V_+}=1$, we define \emph{self-dual} spinors as the sections $F\in\Gamma(S_E)$ satisfying $R_{V_+}F=F$.

\subsection{Spinors and reduction of CAs}\label{sec:spinred}

In this sections we shall summarize how spinors and spinor half-densities behave under reduction of CAs.

If $V$ is a vector space with a non-degenerate symmetric pairing and $J\subset V$ an isotropic subspace then we have a natural isomorphism of spinor spaces
\begin{equation}\label{spinredpt}
S_{J^\perp/J}\otimes(\bw^{top} J)^{1/2}\cong (S_V)^{J\epsilon}\subset S_V.
\end{equation}
where
$$(S_V)^{J\epsilon}:= \{A\in S_V\mid  (\forall u\in J)\, uA=0\}.$$

Let now $E$ be an $\s$-equivariant CA and let $C\s=\s\oplus\s\epsilon$ with $\deg\epsilon=-1$ and $\epsilon^2=0$ be the cone of $\s$ ($C\s$ is a graded Lie algebra, which also has a differential $d/d\epsilon$). We have a natural action of $C\s$ on $S_E$: $\s$ acts on $E$ via $[\chi(s),\cdot]_E$ and thus also on the associated bundle $S_E$, and $\s[1]=\s\epsilon$ acts via $(s\epsilon)\cdot A:=\chi(s) A$. 

If the resulting action of $\mathsf S$ is free and proper, we get from \eqref{spinredpt} (with $V=E$ and $J=\s_E$) the natural isomorphism
\begin{equation}\label{spinred}
\Gamma(S_{E_{/\mathsf S}})\cong\bigl(\Gamma(S_E)\otimes(\bw^\text{top}\s)^{-1/2}\bigr)^{C\s}
\end{equation}
where $\s\subset C\s$ acts on the line $(\bw^\text{top}\s)^{-1/2}$ naturally (via $-\on{Tr}_\s ad_s/2$) and $\s\epsilon$ trivially. A quick inspection shows that if $\dv$ is an equivariant divergence on $E$ and $\widetilde\dv$ the resulting divergence on $E_{/\mathsf S}$ then the isomorphism \eqref{spinred} is compatible with the actions of $D_{\dv}$ and $D_{\widetilde\dv}$.

Pointwise in $M$ the identification \eqref{spinredpt} is
$$S_{\s_E^\perp/\s_E}\cong (S_E)^{\s\epsilon}\otimes (\bw^\text{top}\s)^{-1/2}.$$
The pairing $(\cdot,\cdot)$ on $S_{\s_E^\perp/\s_E}$ transfered under this identification to $(S_E)^{\s\epsilon}\otimes (\bw^\text{top}\s)^{-1/2}$ will be denoted by
$$(\cdot,\cdot)_\s.$$

For spinorial half-densities the situation is somewhat simpler - from \eqref{spinred} we get a natural isomorphism
$$\Gamma(\mathcal S_{E_{/\mathsf S}})\cong\Gamma(\mathcal S_E)^{Cs}$$
compatible with the action of $\mathcal D$, and we get a pairing $(\cdot,\cdot)_\s$ on $(\mathcal S_E)^{s\epsilon}$ with values in $\bw^{top}\s$-valued densities.

\section{Poisson-Lie T-duality and type II SUGRAs}\label{sec:SUGRA}
\subsection{Type II SUGRAs and exact CAs}

The bosonic field content of the type II SUGRAs is 
\[(g, H, \phi, \mathcal{F}),\]
where $g$ is a Lorentzian metric, $H$ is a closed 3-form, $\phi$ is a function (scalar field) called the dilaton and $\mathcal{F}$ (a collection of Ramond-Ramond fields) is an inhomogeneous differential form of even or odd degree satisfying $(d+H)\mathcal F=0$ and a certain self-duality condition. Following \cite{CSCW}, this data is equivalent to $(V_+\subset E,\sigma,\mathcal F)$, where $E\to M$ is an exact CA ($\dim M=10$), $V_+$ is a generalized pseudometric complementary to $T^*M\subset E$, and $\mathcal F\in\Gamma(\mathcal S_E)$ is an even or odd spinor half-density such that $\mathcal D\mathcal F=0$ and $R_{V_+}\mathcal F=\mathcal F$.

Indeed, $V_+\subset E$ is equivalent to $(g,H)$, $\sigma$ is related to $\phi$ via $\sigma=e^{-\phi}\mu_g^{1/2}$ (cf.\ Example \ref{ex:SEF}), and $\Gamma(\mathcal S_E)=\Omega(M)$ and $\mathcal D=d+H$ as in Example \ref{ex:spinex}.

The type II SUGRA equations (or string background equation for a type II superstring) in this formalism are
\begin{subequations}\label{eqdil}
\begin{align}
\gric_{V_+,\dv}(u_+,v_-)\,\sigma^2&=\textstyle{\frac{i}{8\nu}}(u_+ \mathcal F,v_- \mathcal F)\quad \forall u_+\in V_+,v_-\in V_-\\
\Delta_{V_+}\sigma&=0.
\end{align}
\end{subequations}
In the type IIA case $\mathcal F$ is even and $\nu=1$, in the type IIB case $\mathcal F$ is odd and $\nu=i$,

The corresponding ``pseudoaction'' functional is
$$-\frac{1}{2}\int_M\sigma\Delta_{V_+}\sigma-\textstyle{\frac{\nu}{8}}(\mathcal F,R_{V_+}\mathcal F)$$
(where ``pseudo'' corresponds to the fact that the self-duality condition $R_{V_+}\mathcal F=\mathcal F$ has to be imposed manually).

In \emph{modified type II SUGRA} \cite{TW,AFHRT} the dilaton $\phi$ is replaced by a pair $(X,\alpha)$ corresponding to a divergence $\dv$ on $E$ compatible with $V_+$ and $\mathcal F\in\Gamma(\mathcal S_E)$ is replaced by $F\in\Gamma(S_E)$ satisfying $D_{\dv}F=0$ and $R_{V_+}F=F$. The SUGRA equations become
\begin{subequations}\label{eq}
\begin{align}
\gric_{V_+,\dv}(u_+,v_-)&=\textstyle{\frac{i}{8\nu}}(u_+ F,v_- F)\quad \forall u_+\in V_+,v_-\in V_-\\
\mathcal{R}_{V_+,\dv}&=0.\label{eq2}
\end{align}
\end{subequations}

If $\dv=\dv_{\sigma^2}$ (i.e.\ when the dilaton exists) the two systems are linked via
$$\mathcal F=F\,\sigma.$$

	We shall study the systems \eqref{eqdil} end \eqref{eq} for arbitrary (non-exact) CAs and call them \emph{generalized SUGRA equations}. The motivation is to show that PL T-duality is compatible with type II SUGRA, and to actually construct solutions of type II SUGRA. 

To stay close to physics we shall always suppose that the signature of $\la,\ra|_{V_+}$ is Lorentzian and that $\on{rank}V_+\equiv 2$ mod 4 (which includes the physical case of $\on{rank}V_+=10$). In particular, this implies $R_{V_+}^2=1$ and the skew-symmetry of the spinor pairing $(\cdot,\cdot)$.

\subsection{Poisson-Lie T-duality of type II SUGRAs}
We now apply the main idea of PL T-duality, i.e.\ CA-pullbacks, once again:

\begin{thm}
Let $E\to M$ be a CA, $\tau:M'\to M$ a smooth map, and $E'=\tau^*E$ a CA-pullback. If	 $(V_+\subset E,\dv,F)$ is a solution of the generalized type II SUGRA equations, then $(\tau^*V_+\subset E',\tau^*\dv,\tau^*F)$ is also a solution.
\end{thm}
\begin{proof}
We have $D_{\tau^*\dv}\tau^*F=\tau^*(D_{\dv}F)=0$, as follows for example from \eqref{gendirdiv}. The remaining conditions are obvious.
\end{proof}

To get solutions of type II SUGRA using the previous theorem we then need to make sure that $E'$ is exact and (if we want to have a dilaton) that $\tau^*\dv=\dv_\mu$ for a volume form $\mu$ (cf.\ \S\ref{sec:gsbe} and Remark \ref{rem:mugen}).

\subsection{Equivariant PL T-duality of type II SUGRAs}
Finally, let us extend \S\ref{sec:equiv} to generalized type II SUGRA. The main motivation is that it will be a great source of examples, both known and new.

\begin{thm}
Let $E\to M$ be an $\s$-equivariant CA, $\tau:M'\to M$ a smooth map, and $E'=\tau^*E$ a CA-pullback such that the $\s$-action on $E'\to M'$ integrates to a free and proper $\mathsf S$-action. Let $V_+\subset E$ be an admissible generalized pseudometric and $\dv$ an equivariant divergence. Finally, let us choose a section
$$F\in\bigl(\Gamma(S_E)\otimes(\bw^\text{top}\s)^{-1/2}\bigr)^{C\s}.$$

If	 $(V_+\subset E,\dv,F)$ is a solution of the ``equivariant generalized type II SUGRA equations''
\begin{subequations}\label{equivsugra}
\begin{gather}
R_{V_+}F=F,\quad D_{\dv}F=0,\quad \mathcal R_{V_+,\dv}=0\\
\gric_{V_+,\dv}(u_+,v_-)=\tfrac{i}{8\nu}(u_+F,v_-F)_\s\quad\forall u_+\in\Gamma(V_+), v_-\in\Gamma(V_-\cap\s_E^\perp)\label{es_b}
\end{gather}
\end{subequations}
then $(V'_{+,\redu}\subset E'_\redu,\dv'_\redu,F'_\redu)$, obtained from $(V_+\subset E,\dv,F)$ by the pullback followed by the reduction, is a solution of the generalized type II SUGRA equations.
\end{thm}
\begin{proof}
This follows immediately from the previous theorem.
\end{proof}

\begin{example}[PL T-duality for (generalized) type II SUGRA in the case of no spectators] \label{ex:eqPLII}
Let $E=\g$ be a quadratic Lie algebra, $\s\subset\g$ a unimodular and isotropic Lie subalgebra, and let $\dv=0$ (which is $\s$-equivariant as $\s$ is unimodular). Let $V_+\subset\g$ be an admissible generalized pseudometric, i.e.\ $V_+$ is an $\s$-invariant subspace, $\s\perp V_+$, and $\la,\ra|_{V_+}$ is nondegenerate.

Let us fix an isomorphism $\bw^{top}\s\cong\R$  and choose an element 
$$F\in(S_\g)^{C\s}\cong(S_{\s^\perp/\s})^\s.$$
 Let us suppose that  $(V_+\subset\g,\dv=0,F)$ is a solution of Equations \eqref{equivsugra}.

If $\mathsf H\subset\mathsf G$ is a Lie subgroup such that $\h\subset\g$ is a Lagrangian Lie subalgebra,  let  $M'=\mathsf G/\mathsf H$, with the usual exact CA structure on $E'=\g\times \mathsf G/\mathsf H$. Let us suppose that the action of $\mathsf S$ on $\mathsf G/\mathsf H$ is free and proper (in particular, $\s\cap\h=0$). Then $E'_\redu\to\mathsf S\backslash\mathsf G/\mathsf H$ is an exact CA. If $\dim V_+=\dim\g/2-\dim\s$, we thus get a solution of the modified type II SUGRA equations on the background $\mathsf S\backslash\mathsf G/\mathsf H$.\footnote{To be precise (i.e.\ to satisfy all the physical requirements), we need $\dim V_+=10$, the signature of $\la,\ra|_{V_+}$ needs to be Lorentzian, and we also need the transversality condition $\on{Ad}_g\h\cap(\s+V_+)=0$ for all $g\in \mathsf G$ (the last condition is needed to make $V'_{+,\redu}\subset E'_\redu$  correspond to a pair $(g,H)$, and it is needed because $\la,\ra|_{V_+}$ is indefinite).}

Moreover, if $\h$ is unimodular, it is a solution of the ordinary type II SUGRA equations; the dilaton corresponds to an invariant half-density on $\mathsf G$ pushed down to $\mathsf S\backslash\mathsf G/\mathsf H$ using invariant half-densities on $\mathsf H$ and $\mathsf S$.
\end{example}

\section{Examples: symmetric spaces}\label{sec:ex}

In this section we shall find various solutions of (generalized) type II SUGRA equations using the general setup of Example \ref{ex:eqPLII}. Namely, we shall get solutions on suitable symmetric spaces. The reason to look at this type of examples is that they are quite easy, and also have nice properties. The corresponding 2-dim $\sigma$-models are completely integrable systems (as shown in \cite{K} for groups and in \cite{DMV0} in generality). A 1-parameter family, $\eta$-deformed $AdS_5\times S^5$, found in \cite{DMV}, was shown in \cite{AFHRT} to be a solution of modified type IIB SUGRA equations. With our method we find many other families, some of them with several parameters.
\subsection{Building blocks}

Let $\a$ be a semisimple real Lie algebra with the invariant inner product $K=\mathcal{K}/\lambda$, where $\lambda\in\mathbb{R}$ and $\mathcal{K}$ is the Killing form of $\a$. Suppose furthermore that $\a$ is equipped with an involution, i.e.\ that we have an (orthogonal) splitting $\a=\a_0\oplus\a_1$ such that $[\a_0,\a_0]\subset\a_0$, $[\a_1,\a_1]\subset\a_0$, $[\a_0,\a_1]\subset\a_1$.

Consider the commutative algebra $\mathcal{B}_c:=\mathbb{R}[t]/(t^2-c)$ for a fixed $c\in\R$. We set 
$$\g:=\mathcal{B}_c\otimes \a \quad \text{with}\quad [p\otimes u,q\otimes v]:=pq\otimes [u,v]$$
and we define the inner product 
$$\la p \otimes u,q \otimes v \ra=(pq)_\text{lin}K(u,v)$$
 where $(a+bt)_\text{lin}:=b$ ($\forall a,b\in\R$). For simplicity we will omit the '$\otimes$' symbol from now on.
Finally, we take $$\s:=\a_0\subset\a\subset\g\quad \text{and}\quad V_+:=(1+t)\a_1.$$ This implies $V_-=(1-t)\a_1\oplus(\mathcal{B}_c\otimes\a_0)$.

To construct the spinors, we take the lagrangian decomposition of $\g$ in the form $L_1=t\a$, $L_2=\a$.  We identify $L_1\cong\a\cong\a^*$ via $K$ and $(\bw^{\mathrm{top}}L_1)^{-1/2}$ as well as $(\bw^{\mathrm{top}}\s)^{-1/2}$ with $\R$ using the volume forms associated to $K$ and $K|_{\a_0}$, respectively. Thus 
$$S_\g\cong\bw \a^*, \quad S_{\s^\perp/\s}\cong \bw \a_1^*,\quad \text{and} \quad (S_\g)^{C\s}\cong(\bw \a_1^*)^{\a_0}.$$
We will use the notation $j_v$ for the endomorphism of $S_\g\cong \bw \a^*$ given by
$w\mapsto v\wedge w$
for $v\in \a\cong \a^*$.

\begin{lem}
For this setup, the Dirac generating operator $D_0$ vanishes on $(S_\g)^{C\s}$.
\end{lem}
\begin{proof}
Let $E_\alpha$ be a basis and $f\in\wedge^3\a_0+(\wedge^2 \a_1)\wedge \a_0$ the structure constants of $\a$. Since
$$\la [t\a,t\a],\a\ra=\la [\a,\a],\a\ra=0,$$
the expression for $D_0$ reduces to 
\[D_0=-\frac{1}{2}f^{\alpha\beta\gamma}j_{E_\alpha}j_{E_\beta}\iota_{E_\gamma}-\frac{c}{6}f^{\alpha\beta\gamma}\iota_{E_\alpha}\iota_{E_\beta}\iota_{E_\gamma}=d_{CE}-c\,\iota_f,\]
where $d_{CE}$ is the Chevalley-Eilenberg differential.
One now easily sees that both terms vanish separately on $(\bw \a_1^*)^{\a_0}\subset \bw \a^*$. 
\end{proof}

\begin{lem}\label{algebraic_ric}
We have
\[\gric_{V_+,0}(V_+,\a_0)=0\]
and
\[\gric_{V_+,0}((1+t)u,(1-t)v)=\textstyle{\frac{c-1}{2}}\mathcal{K}_\a(u,v)\]
for all $u,v\in \a_1$.
\end{lem}
\begin{proof}
The first claim is easy to check (cf.\ \eqref{gricvan}).
For the second one, we define
\[p\colon \a_1\to V_+\subset \g,\quad u\mapsto (1+t)u.\]
Then we have
\begin{align*}
-\gric_{V_+,0}&((1+t)u,(1-t)v)=\on{Tr}_{V_+}[[\,\cdot\,,(1-t)v]_-,(1+t)u]_+\\
&=\on{Tr}_{V_+}[[\,\cdot\,,(1-t)v],(1+t)u]_+=\on{Tr}_{\a_1}p^{-1}\bigl([[p(\cdot),(1-t)v],(1+t)u]_+\bigr)\\
&=(1-c)\on{Tr}_{\a_1}[[\,\cdot\,,v],u].
\end{align*}
Here in the second step we used $[V_+,V_+]_+=0$.

To finish, let us now use the Latin indices for the elements of a basis $E_a$ of $\a_1$, the Greek indices for a basis $E_\alpha$ of $\a$ and denote again the structure constants of $\a$ by $f$. We also set $E^\alpha:=K^{\alpha \beta}E_\beta$ (the dual basis). Since 
\begin{equation}\label{gradovana_identita}
K([\a_1,\a_1],\a_1)=K([\a_0,\a_0],\a_1)=0,
\end{equation}
we have
\[\mathcal{K}_{\a}(E_a,E_b)=-f_{a\gamma\delta} f_{b}^{\,\,\gamma\delta}=-f_{ac\delta} f_{b}^{\,\,c\delta}-f_{a\gamma d} f_{b}^{\,\,\gamma d}=-2f_{ac\delta} f_{b}^{\,\,c\delta}=2\on{Tr}_{\a_1}[[\cdot,E_a],E_b],\]
implying $\on{Tr}_{\a_1}[[\,\cdot\,,v],u]=\frac{1}{2}\mathcal{K}_{\a}(u,v)$.
\end{proof}

\begin{lem}\label{biggamma}
\[\mathcal{R}_{V_+,0}=\textstyle{\frac{1+c}{4}}\lambda\dim \a_1.\]
\end{lem} 
\begin{proof}
Adopting the notation from the previous proof and using (\ref{gradovana_identita}) we have 
\[f_{a\beta\gamma}f^{a\beta\gamma}=2f_{ab\gamma}f^{ab\gamma}\]
and thus
\begin{align*}
\mathcal{R}_{V_+,0}&=-\textstyle{\frac{1}{8}}\la [(1+t)E_a,(1+t)E_b],[(1+t)E^a,(1+t)E^b]\ra\\
&=-\textstyle{\frac{1+c}{2}}f_{ab\gamma}f^{ab\gamma}=-\textstyle{\frac{1+c}{4}}f_{a\beta\gamma}f^{a\beta\gamma}=\textstyle{\frac{1+c}{4}}\mathcal{K}(E_a,E^a)=\textstyle{\frac{1+c}{4}}\lambda\dim \a_1.\qedhere
\end{align*}
\end{proof}

\subsection{Putting the blocks together}\label{blocks_together}

We now restrict our attention to the ten-dimensional case and try to construct solutions to the (generalized) SUGRA in the form $AdS_m\times X$ for $X$ a Riemannian symmetric space. Notice that $AdS_m$ is itself a (Lorentzian) symmetric space, $AdS_m=O(m-1,2)/O(m-1,1)$, of dimension $m$.

Let $\a^{(k)}$, $k=1,\dots, N$, be some real compact simple Lie algebras with invariant inner products $K_k=\lambda_k^{-1}\mathcal{K}_k$, where $\mathcal{K}_k$ are the respective Killing forms. Suppose also that all these algebras are equipped with involutions,\footnote{We exclude the non-interesting cases when the involutions are equal to the identity.} inducing the splittings $\a^{(k)}=\a^{(k)}_0\oplus \a^{(k)}_1$. In addition, we set $$\a^{(0)}:=\mathfrak{o}(m-1,2),\quad \a^{(0)}_0:=\mathfrak{o}(m-1,1)$$ for a fixed $m\in\{2,\dots,8\}$, we choose $c_k\in\R$ for $k=0,1,\dots, N$ and define
$$\g^{(k)}:=\mathcal{B}_{c_k}\otimes \a^{(k)},\quad [p\otimes u,q\otimes v]:=pq\otimes [u,v],\quad \la p \otimes u,q \otimes v \ra=(pq)_{\text{lin}}K_k(u,v).$$
We set\footnote{This fixes the overall scaling freedom present in the SUGRA equations of motion.} $\lambda_0=1$ and impose that $\lambda_k<0$ for $k\neq 0$.

We also allow the possibility of adding an ``abelian part'' to the story, corresponding to tori in the resulting symmetric space. Thus, let $\b$ be an abelian Lie algebra with a positive definite inner product $K_\b$. We make the (trivial) split $\b=\b_0\oplus \b_1$ by taking $\b_0=0$, $\b_1=\b$ and define the quadratic abelian Lie algebra $\g^{(\b)}:=\mathcal{B}_0\otimes \b$ as above.

As the generalized pseudometrics we take 
\[ V_+^{(k)}:=(1+t)\a_1^{(k)}\subset\g^{(k)},\quad V_+^{(\b)}:=(1+t)\b\subset\g^{(\b)}.\]
Finally, we set
\[\g:=\g^{(0)} \oplus \dots \oplus \g^{(N)} \oplus \g^{(\b)},\]
together with
\[V_+:=V_+^{(0)}\oplus \dots \oplus V_+^{(N)}\oplus V_+^{(\b)}\subset \g,\]
and 
\[\s:=\a^{(0)}_0 \oplus \dots \oplus \a^{(N)}_0\subset \g\]
and we suppose that $\dim V_+=10$. Notice in particular that the inner product on $V_+$ is Lorentzian (since $K|_{\a^{(0)}_1}$ has the signature $-+\dots+$, $K_\b$ is positive definite, and the remaining $\a^{(k)}$'s have negative definite Killing forms). 
Let us use the notation
\[\a^{\text{total}}:=\a^{(0)} \oplus \dots \oplus \a^{(N)}\oplus \b\subset \g\quad \text{and}\quad \a^{\text{total}}_1:=\a^{(0)}_1 \oplus \dots \oplus \a^{(N)}_1\oplus \b_1\subset \g.\]

Concerning spinors, we have $\g=L_1\oplus L_2$, $L_1=t\a^{\text{total}}$, $L_2=\a^{\text{total}}$.  Identifying again $L_1\cong\a^{\text{total}}\cong(\a^{\text{total}})^*$ and $(\bw^{\mathrm{top}}L_1)^{-1/2}\cong\R\cong(\bw^{\mathrm{top}}\s)^{-1/2}$ via the corresponding inner products and volume forms respectively, we have 
$$S_{\s^\perp/\s}\cong \bw (\a_1^{\text{total}})^*\quad \text{and} \quad (S_\g)^{C\s}\cong(\bw \a_1^{\text{total}*})^{\s}$$
and also \[(A,B)_\s=(A,B)=((\vartheta A) \wedge B)^{\text{top}},\quad \forall A,B\in\bw (\a_1^{\text{total}})^*.\]

\subsection{The generalized SUGRA equations}
First, let us discuss the self-duality condition for spinors.
\begin{lem}\label{manysigns}
Let $*$ denote the Hodge operator on $\bw (\a_1^{\text{total}})^*$ induced by $K$.
Then 
\[R_{V_+}=*\nu\vartheta,\]
where $\vartheta\colon \bw (\a_1^{\text{total}})^*\to \bw (\a_1^{\text{total}})^*$ is the operator $\vartheta \xi=i^{|\xi|}\xi$, and
$\nu$ is equal to $1$ for even degree spinors and $i$ for odd spinors. 
\end{lem}
\begin{proof}
Using the identification $(\a_1^{\text{total}})^*\cong\a_1^{\text{total}}$ and an orthonormal basis $E_a$ of $\a_1^{\text{total}}$, we may write for $\xi\in\wedge^k (\a_1^{\text{total}})^*$
\begin{align*}
R_{V_+}\xi= [\textstyle{\prod_{a=1}^n(1+t) E_a]\xi}=[\textstyle{\prod_{a=1}^n(\iota_{E_a}+j_{E_a})]\xi}=(-1)^{\left[\frac{k+1}{2}\right]+kn}*\xi.
\end{align*}
Since $(-1)^{\left[\frac{k+1}{2}\right]+kn}\xi=\nu\vartheta\xi$, the lemma follows.
\end{proof}

To deal with the RHS of (\ref{es_b}), let us choose $F\in (\bw (\a_1^{\text{total}})^*)^{\s}\cong (S_\g)^{C\s}$ such that $R_{V_+}F=F$ and define $\psi_F\in (\a^{\text{total}\,*}_1)^{\otimes 2}$ by
\[\psi_F(u,v):=\textstyle{\frac{i}{4\nu}}((1+t)u F, (1-t)v F),\quad u,v\in \a^{\text{total}}_1.\]
Notice that $\psi_F$ is $\s$-invariant as a consequence of the invariance of $F$. It can be conveniently computed as follows:

\begin{lem}\label{psicko}
For $u,v\in \a^{\text{total}}_1$
\[\psi_F(u,v)=\textstyle{\frac{1}{4}}([(\iota_u+j_u)(\iota_v-j_v)F]\wedge *F)^{\text{top}}.\]
\end{lem}
\begin{proof}
We have $((\vartheta A)\wedge *B)^{\text{top}}=(A\wedge *\vartheta B)^{\text{top}}$ for all $A,B\in \bw (\a_1^{\text{total}})^*$ and $(\vartheta^2 A,B)=(A,\vartheta^2 B)$ since $\dim \a^{\text{total}}_1=10$. Therefore
\begin{multline*}
((1+t)u F, (1-t)v F)=i(\vartheta^2 F, (u+tu)(v-tv) F)=-i(\vartheta^2 (u+tu)(v-tv) F,F)\\
=-i(\vartheta^2 (u+tu)(v-tv) F,*\nu\vartheta F)=-i\nu([(\iota_u+j_u)(\iota_v-j_v)F]\wedge *F)^{\text{top}}.
\end{multline*}
\end{proof}

 Then, putting things together, we obtain the following.

\begin{thm}
In the setup of \S\ref{blocks_together}, the generalized SUGRA equations \eqref{equivsugra} are equivalent to
\begin{subequations}\label{equivsugraalgebraic}
\begin{align}
(1+c_0)\dim \a^{(0)}_1&=\textstyle{\sum_{k=1}^N} (-\lambda_k) (1+c_k) \dim \a^{(k)}_1,\label{first_algebraic}\\
(c_k-1)\mathcal{K}_{\a^{(k)}}(u,v)&=\psi_F(u,v),\quad \forall u,v\in \a^{(k)}_1,\quad k=0,\dots, N\label{second_algebraic}\\
\psi_F(\b,\b)&=0,\label{third_algebraic}
\end{align}
\end{subequations}
where $F\in (\bw (\a_1^{\text{total}})^*)^{\s}$ is self-dual ($R_{V_+}F=F$).
\end{thm}
\begin{proof}
Recall that $V_-\cap\s^\perp=\s\oplus\bigoplus_{k=0}^N(1-t)\a^{(k)}_1$.
Due to the $\s$-invariance of $F$, the RHS of equation (\ref{es_b}) vanishes for $u_+\in V_+, v_-\in \s$.

The fact that \[\psi_F(\b,\a^{\text{total}}_1)=\psi_F(\a^{\text{total}}_1,\b)=\psi_F(\a^{(k)}_1,\a^{(l)}_1)=0,\quad \forall k  \neq l\]
follows from the $\s$-invariance of $\psi_F$. Indeed, if $u\in\a^{(k)}_1$ (resp. $u\in\b$), then for any $l\neq k$ the element $\psi_F(u,\cdot)|_{\a^{(l)}_1}\in(\a^{(l)}_1)^*$ is $\a^{(l)}_0$-invariant and hence vanishes.

The proof then follows from lemmas (\ref{algebraic_ric}), (\ref{biggamma}), (\ref{psicko}), the fact that $\b$ is abelian and $\g=\bigoplus_{k=0}^N\g^{(k)}\oplus\g^{(\b)}$ is an orthogonal decomposition.
\end{proof}

\subsection{Constructing exact Courant algebroids}
Let us construct a connected Lie group $\mathsf G$ integrating $\g$ as follows: For every $k=0,\dots,N$ let $\mathsf G^{(k)}$ be a connected group integrating $\g^{(k)}$ and let $\mathsf G^{(\b)}$ be a torus integrating $\g^{(\b)}$ (i.e.\ $\mathsf G^{(\b)}=\g^{(\b)}/\Lambda$ for some lattice $\Lambda$; we choose $\mathsf G^{(\b)}$ to be compact because of physical considerations). We then set $\mathsf G=\mathsf G^{0}\times\dots\times\mathsf G^{(N)}\times\mathsf G^{(\b)}$. Let $\mathsf S\subset \mathsf G$ be the connected group integrating $\s$. 

We now need to choose a closed subgroup $\mathsf H\subset\mathsf G$ such that $\h^\perp=\h$  and such that $\mathsf S$ acts freely and properly on $\mathsf G/\mathsf H$. There is no canonical choice of $\mathsf H$. We shall use the following  $\mathsf H$ which can be defined if $c_k\leq 0$ for each $k$ (but other choices are possible):

If $c_k=0$, let $\h^{(k)}=t\a^{(k)}\subset \g^{(k)}$ and let $\mathsf H^{(k)}\subset\mathsf G^{(k)}$ be the corresponding (abelian) connected Lie group.

If $c_k<0$ then $\mathcal{B}_{c_k}\cong \C$ and thus $\g^{(k)}\cong\a^{(k)}\otimes\C$. Let $\h^{(k)}$  be the direct sum of the real Cartan subalgebra\footnote{i.e.\ the real span of the coroots} and the nilpotent (complex) Lie subalgebra spanned by the negative root spaces.  Notice that $\h^{(k)}$ is not unimodular in this case. Again $\mathsf H^{(k)}\subset\mathsf G^{(k)}$ is the corresponding connected Lie subgroup. For $c_k\leq 0$ the subgroups $\mathsf H^{(k)}\subset\mathsf G^{(k)}$ form a continuous family (parametrized by $c_k$).

Finally, let $\mathsf H^{(\b)}\subset \mathsf G^{(\b)}$ be a torus such that $\h^{(\b)}\subset \g^{(\b)}$ is lagrangian. We set
$\mathsf H=\mathsf H^{(0)}\times\dots\times\mathsf H^{(N)}\times\mathsf H^{(\b)}$ and we get (using the Iwasawa decomposition $\mathsf G^{(i)}=\mathsf A^{(i)}_{\C}=\mathsf A^{(i)}\mathsf H^{(i)}$ for $i\geq 1$ if $c_i<0$)
\begin{equation}
\label{resulting_symmetric_space}\mathsf{S}\backslash \mathsf{G}/\mathsf{H}\cong AdS_m\times \mathsf{A}^{(1)}/\mathsf{A}_0^{(1)}\times \dots\times \mathsf{A}^{(N)}/\mathsf{A}_0^{(N)}\times (S^1)^{\dim \b},
\end{equation}
where $\mathsf{A}^{(k)}$ and $\mathsf{A}_0^{(k)}$ are Lie groups integrating $\a^{(k)}$ and $\a_0^{(k)}$, respectively. If $c_k=0$ for every $k$ then the generalized metric $V_+$ gives us  a pseudo-Riemannian metric $g$ making this manifold to a symmetric pseudo-Riemannian space, and the 3-form $H$ vanishes. For general $c_k$'s the metric $g$ (and also the 3-form $H$) is quite different, but it depends continuously on $c_k$'s.

\subsection{First ansatz}
As a first example, let us consider $N=1$, \[\a^{(1)}=\mathfrak{su}(M+1),\quad \a^{(1)}_0=\mathfrak{s}(\mathfrak{u}(1)\oplus\mathfrak{u}(M)),\] for a fixed $M\in\{1,2,3\}$, implying $\mathsf{S}\backslash \mathsf{G}/\mathsf{H}=AdS_{10-2M}\times \C \text{P}^M$. 
 The space $\wedge^2 (\a_1^{(1)})^*$ is spanned by the element\footnote{This corresponds (up to a constant) to the symplectic form on $\C \text{P}^M$.}
\[\Omega=e_1\wedge e_2+\dots+e_{2M-1}\wedge e_{2M}\]
for a suitable orthonormal basis (w.r.t.\ $K_1$) $e_a$ of $\a^{(1)}_1$
and a general self-dual element of $(\bw (\a_1^{\text{total}})^*)^{\s}$ is of the form
\[F=p(\Omega)+R_{V_+}p(\Omega)=p(\Omega)+*\vartheta p(\Omega),\]
for $p$ a polynomial,
\[p(\Omega)=\sum_{n=0}^M \frac{d_n}{n!}\Omega^n,\quad d_n\in\R.\]

We omit the proof of the following proposition (it is just a straightforward calculation).
\begin{prop}
Equations (\ref{equivsugraalgebraic}) reduce to
\begin{align*}
(5-M)(1+c_0)&=-\lambda_1 M(1+c_1),\\
2(1-c_0)&=\sum_{n=0}^M d_n^2{M \choose n},\\
-\lambda_1 2M(1-c_1)&=\sum_{n=0}^M d_n^2{M \choose n}(2n-M),\\
0&=\sum_{n=1}^M d_n d_{n-1}{{M-1} \choose {n-1}}.
\end{align*}
\end{prop}

Notice that there are $M+4$ free parameters constrained by 4 equations, giving an $M$-parameter class of solutions.

\subsection{Second ansatz}
We shall now construct solutions where $F$ is built only out of volume forms. 

Let us denote by $\omega_k$ and $\omega_\b$ the metric volume forms on $\a^{(k)}_1$, for $k=0,\dots, N$ and on $\b=\b_1$, respectively. Let us consider a general element of 
$(S_\g)^{C\s}\cong(\bw \a_1^{\text{total}*})^{\s}$ which is a linear combination of products of $\omega_i$'s:
\[\hat F:=\!\!\!\!\!\!\!\sum_{h\in \{0,1\}^{N+2}} \!\!\!\!\!\! f(h)\,\omega_0^{h(0)}\wedge\dots\wedge \omega_N^{h(N)}\wedge\omega_\b^{h(N+1)}\]
(for arbitrary coefficients $f(h)\in\R$)
or, if $\b=0$,
\[\hat F:=\!\!\!\!\!\!\!\sum_{h\in \{0,1\}^{N+1}} \!\!\!\!\!\! f(h)\,\omega_0^{h(0)}\wedge\dots\wedge \omega_N^{h(N)}.\]
Let us then set
$$F:=\hat F+R_{V_+}\hat F.$$
Clearly, $F$ is self-dual, $\s$-invariant, and thus $D_0$-closed.

\begin{lem}\label{biglemma}
If
\begin{equation}\label{diota}
(\iota_u \iota_v F\wedge * F)^\text{top}=0,\quad \forall u,v\in \a_1^{\text{total}},
\end{equation}
then the conditions \eqref{second_algebraic}, \eqref{third_algebraic} are equivalent to
\begin{equation}\label{eq4}
2(1-c_k)\lambda_k=\sum_{h}(-1)^{h(0)+h(k)}f(h)^2
\end{equation}
for $k=0,\dots,N$ if $\b=0$ or $k=0,\dots,N+1$ if $\b\neq 0$,
where we define $\lambda_{N+1}:=0$.
\end{lem}
\begin{proof}
Let us consider the case $\b=0$; the case of $\b\neq0$ is similar. The condition (\ref{diota}) implies also $(j_u j_v F\wedge * F)^\text{top}=0$ and thus, taking $u,v\in \a^{(k)}_1$ for $k\in\{0,\dots, N\}$,
\begin{align*}
\psi_F(u,v)=\tfrac{1}{4}([(j_u\iota_v-\iota_u j_v) F]\wedge*F)^{\text{top}}.
\end{align*}
An easy calculation now shows that this expression equals
\begin{align*}
\tfrac{1}{2}&([(j_u\iota_v-\iota_u j_v) \hat F]\wedge*\hat F)^{\text{top}}=\tfrac{1}{2}\!\!\sum_{h(k)=1} f(h)([j_u \iota_v\omega_0^{h(0)}\wedge\dots\wedge \omega_N^{h(N)}]\wedge * \hat F)^{\text{top}}\\
&-\tfrac{1}{2}\!\!\sum_{h(k)=0} f(h)([\iota_u j_v\omega_0^{h(0)}\wedge\dots\wedge \omega_N^{h(N)}])\wedge * \hat F)^{\text{top}}\\
&=-\tfrac{1}{2}\sum_h (-1)^{h(k)}f(h)^2K(u,v) \bigl((\omega_0^{h(0)}\wedge\dots\wedge \omega_N^{h(N)})\wedge *(\omega_0^{h(0)}\wedge\dots\wedge \omega_N^{h(N)})\bigr)^{\text{top}}\\
&=-\tfrac{1}{2}K(u,v) \sum_h (-1)^{h(0)+h(k)}f(h)^2,
\end{align*}
where in the next-to-last equality we have used that 
\[(\omega_0^{h(0)}\wedge\dots\wedge \omega_N^{h(N)})\wedge *(\omega_0^{h(0)}\wedge\dots\wedge \omega_N^{h(N)})=(-1)^{h(0)}\]
due to the Lorentzian signature of the inner product on $\a^{(0)}_1$.
\end{proof}

Putting things together, we obtain:

\begin{prop}\label{miniprop}
For this setting, supposing the conditions \eqref{first_algebraic}, \eqref{diota}, \eqref{eq4} hold and the spinor $F$ is of a definite parity, we obtain a solution to the generalized SUGRA equations of motion on the manifold \eqref{resulting_symmetric_space}.
\end{prop}

\subsection{Second ansatz - examples}
We now present several interesting special cases of the above construction. Since it is easy to check that in the following examples (\ref{diota}) holds, we will only focus on the analysis of the conditions (\ref{first_algebraic}) and (\ref{eq4}).
\subsubsection{}\label{easiest}
Let us first consider the case $N=1$, $\b=0$ with
\[f(h)=
\begin{cases}
  a, & \text{if } h=(1,0)\\
  0, & \text{otherwise}.
\end{cases}\]
This corresponds to $\a=\mathfrak{o}(m-1,2)\oplus \a^{(1)}$ with 
\[F=a(\omega_0+R_{V_+}\omega_0)=a(\omega_0-\nu i^m \omega_1).\]
Here $m\in\{2,\dots,8\}$. The conditions (\ref{first_algebraic}) and (\ref{eq4}) then give
\[(1+c_0)m=-\lambda_1(1+c_1)(10-m),\]
\[2(1-c_0)=a^2,\qquad 2(1-c_1)\lambda_1=-a^2,\] 
which can be recast conveniently as
\begin{equation*}
\frac{1+c_0}{1-c_0}\frac{m}{10-m}=\frac{1+c_1}{1-c_1},\quad \lambda_1=-\frac{1-c_0}{1-c_1},\quad a^2=2(1-c_0),
\end{equation*}
and yields a one-parameter family of solutions (parametrized e.g.\ by $c_0$). 
We thus have

\begin{thm}
For each real compact semisimple Lie algebra $\a(=\a^{(1)})$ with an involution ($\a=\a_0\oplus \a_1$), such that $\dim \a_1\in \{2,\dots,8\}$, we obtain, via the above construction, a one-parameter family of solutions to the generalized SUGRA equations of motion on the 10-dimensional Lorentzian manifold $AdS_m\times \mathsf{A}/\mathsf{A}_0$. 
\end{thm}

Here we provide a list of irreducible\footnote{in the sense of Riemannian symmetric spaces} examples, writing $\mathsf{A}/\mathsf{A}_0$ and the dimension of the quotient for each case.
\begin{center}
\begin{tabular}{ c c }
 $S^d=SO(d+1)/SO(d)$ & d\\ 
 $\C \mathrm{P}^d=U(d+1)/(U(1)\times U(d))$ & 2d\\  
 $SU(4)/S(U(2)\times U(2))$ & 8\\  
 $Sp(2)/U(2)$ & 6\\
 $SO(6)/(SO(2)\times SO(4))$ & 8\\
 $G_2/SO(4)$ & 8\\
 $SU(3)/SO(3)$ & 5\\
 $SU(3)=(SU(3)\times SU(3))/SU(3)$ & 8\\
\end{tabular}
\end{center}
Here in the first entry $d\in\{2,\dots,8\}$ and in the second $d\in\{2,3,4\}$. The next three examples (as well as the projective spaces) are Grassmannian spaces. For reducible examples, we can consider the products of elements of the above list, e.g. $S^2\times S^2 \times S^3$ or $S^2\times Sp(2)/U(2)$.

\begin{rem}
Notice that in the case $m=\dim \mathsf{A}/\mathsf{A_0}=5$ the family of solutions simplifies to $c_0\in\R$, $c_1=c_0$, $\lambda_1=-1$, $a^2=2(1-c_0)$. It also contains $c_0=c_1=0$, which corresponds to a solution to (ordinary) supergravity equations. This one-parameter class recovers in particular the $\eta$-deformed $AdS_5\times S^5$, found in \cite{DMV}. In addition, from the above list we also get the cases $AdS_5\times S^3\times S^2$ and $AdS_5\times SU(3)/SO(3)$.
\end{rem}

\subsubsection{}
Let us again take the case $N=1$, $\b=0$ but now with $m\in\{3,\dots,7\}$ and
\[f(h)=
\begin{cases}
  a, & \text{if } h=(1,0)\\
  b, & \text{if } h=(0,0)\\
  0, & \text{otherwise}.
\end{cases}\]
This corresponds to $\a=\mathfrak{o}(m-1,2)\oplus \a^{(1)}$ with 
\[F=a(\omega_0+R_{V_+}\omega_0)+b(1+R_{V_+}1)=a(\omega_0- i^m \omega_1)+b(1+\omega_0\wedge \omega_1).\] 
The conditions (\ref{first_algebraic}) and (\ref{eq4}) give
\[(1+c_0)m=-\lambda_1(1+c_1)(10-m),\]
\[2(1-c_0)=a^2+b^2,\qquad 2(1-c_1)\lambda_1=-a^2+b^2,\] 
which yields a two-parameter class of solutions. Therefore

\begin{thm}
For each real compact semisimple Lie algebra $\a(=\a^{(1)})$ with an involution ($\a=\a_0\oplus \a_1$), such that $\dim \a_1\in \{3,\dots, 7\}$, we obtain, via the above construction, a two-parameter family of solutions to the generalized SUGRA equations of motion on the 10-dimensional Lorentzian manifold $AdS_m\times \mathsf{A}/\mathsf{A}_0$.
\end{thm}

\subsubsection{}
We shall illustrate the case of $\b\neq 0$ on the example of $N=1$, $\b=\mathfrak{u}(1)$ for $AdS_4\times SU(3)/SO(3)\times S^1$ with 
\[f(h)=
\begin{cases}
  a, & \text{if } h=(1,0,0)\\
  b, & \text{if } h=(0,1,0)\\
  d, & \text{if } h=(0,0,1)\\
  0, & \text{otherwise},
\end{cases}\]
corresponding to 
\[F=a(\omega_0- \omega_1\wedge \omega_2)+b(\omega_1-\omega_0\wedge \omega_2)+d(\omega_2+\omega_0\wedge \omega_1).\]
The equations (\ref{first_algebraic}) and (\ref{eq4}) now read
\[4(1+c_0)=-5\lambda_1(1+c_1),\]
\[2(1-c_0)=a^2+b^2+d^2,\qquad 2(1-c_1)\lambda_1=-a^2-b^2+d^2,\qquad 0=-a^2+b^2-d^2,\]
producing again a two-parameter class of solutions.

\begin{rem}
It is clear that we can continue and consider other cases: $N>1$, $\dim \b>1$, \dots In this way we can produce other classes of solutions, some of them having more than 2 free parameters.
\end{rem}

\end{document}